\documentclass[10pt,a4paper]{article}

\usepackage[english]{babel}
\usepackage[T1]{fontenc}
\usepackage{bbm}
\usepackage{graphicx}
\usepackage{verbatim}
\usepackage{amsmath}
\usepackage{amssymb}
\usepackage{amsthm}
  \theoremstyle{plain}
  \newtheorem{theorem}{Theorem}
  \newtheorem{lemma}{Lemma}
  
  \newtheorem{corollary}{Corollary}
  \newtheorem{assumption}{Assumption}

  \theoremstyle{remark}
  \newtheorem{remark}{Remark}
\usepackage{mathrsfs}
\usepackage{mathtools}
\usepackage[left=2.7cm,right=2.7cm,top=3.1cm,bottom=3.1cm]{geometry}
\usepackage[colorlinks=true, allcolors=blue, hypertexnames=false]{hyperref}
\usepackage{subfigure}
\usepackage{graphicx}
\usepackage{algpseudocode,float,algorithmicx,algorithm}
\usepackage[numbers]{natbib}
\makeatletter
\newenvironment{breakablealgorithm}
{% \begin{breakablealgorithm}
	\begin{center}
		\refstepcounter{algorithm}% New algorithm
		\hrule height.8pt depth0pt \kern2pt% \@fs@pre for \@fs@ruled
		\renewcommand{\caption}[2][\relax]{% Make a new \caption
			{\raggedright\textbf{\ALG@name~\thealgorithm} ##2\par}%
			\ifx\relax##1\relax % #1 is \relax
			\addcontentsline{loa}{algorithm}{\protect\numberline{\thealgorithm}##2}%
			\else % #1 is not \relax
			\addcontentsline{loa}{algorithm}{\protect\numberline{\thealgorithm}##1}%
			\fi
			\kern2pt\hrule\kern2pt
		}
	}{% \end{breakablealgorithm}
		\kern2pt\hrule\relax% \@fs@post for \@fs@ruled
	\end{center}
}
\makeatother

 \allowdisplaybreaks[4]
\pdfoutput=1
\begin{document}

\title{Policy Optimization of Finite-Horizon Kalman Filter with Unknown Noise Covariance\thanks{This work is supported in part by the National Natural Science Foundation of China (62173191).}}
\author{Haoran Li\thanks{College of Artificial Intelligence, Nankai University, Tianjin 300350, P.R. China. Email: {\tt 2768691501@qq.com}.}  \and Yuan-Hua Ni\thanks{College of Artificial Intelligence, Nankai University, Tianjin 300350, P.R. China. Email: {\tt yhni@nankai.edu.cn}.}}
	
\date{\today}

\maketitle

\begin{abstract}
This paper is on learning the Kalman gain by policy optimization method. Firstly, we reformulate the finite-horizon Kalman filter   as a policy optimization problem of the dual system. Secondly, we obtain the global linear convergence of exact gradient descent method in the setting of known parameters. Thirdly, the gradient estimation and stochastic gradient descent method are proposed to solve the policy optimization problem, and further the global linear convergence and sample complexity of stochastic gradient descent are provided for the setting of unknown noise covariance matrices and known model parameters.   


\end{abstract}

\section{Introduction}

Kalman filter (KF) is widely used in the fields of control and estimation, such as in road vehicle tracking \cite{6426451,HABIB201311,7015597}, spacecraft orbit estimation and control \cite{197900}, active power filter \cite{4538497}, medical diagnosis \cite{5521169}, and ship control \cite{197901}. However, the noise covariance matrices are usually unknown, and this makes it difficult to use classical KF in practical problems. To solve this puzzle, many approaches have been proposed \cite{1100100}, and there are mainly four types: Bayesian inference, maximum likelihood estimation, innovation correlation and noise variance matching. Bayesian inference uses recursive method to calculate the posterior estimation of unknown parameters under observation, and then calculates the joint posterior probability of state and unknown parameters \cite{4082201,10.1115/1.3662552}; 
maximum likelihood estimation uses nonlinear programming method to calculate the optimal probability density function \cite{1180077,Smith1971EstimationOT}. %Both of these methods have higher requirements for computation. 
Covariance matching technique is to make the covariance of  innovation sequence consistent with its theoretical value \cite{1101260}. Innovation correlation method {uses the system output to generate a set of equations related to the model parameters}, and solves the these equations to obtain the estimations of unknown parameters; this method has been extensively studied in recent year \cite{1100694,1624478,ODELSON2006303,9044358}. Interestingly, most of the analysis of these methods is based on the {asymptotic viewpoint}.

In this paper, we consider to solve a finite-horizon KF problem (with unknown covariance matrices) from the perspective of policy optimization, namely, to learn the Kalman gain by the {stochastic gradient descent (SGD) method}. Compared with the aforementioned works, we focus on exploring the global convergence, convergence rate and sample complexity of the SGD method in the policy optimization landscape. 
The basis of our research is the duality between control and estimation, and the recent developments of policy gradient (PG) method. {In the areas of optimal control and reinforcement learning}, zeroth-order PG method has been specifically analyzed. The work  \cite{pmlr-v80-fazel18a} reformulates the classical linear-quadratic regulator (LQR) as an optimization problem of  feedback gain, and obtain the gradient dominance condition and almost smoothness of optimization function by using the advantage function. By the two properties and zeroth-order optimization technique, global linear convergence of {gradient descent (GD) method} is established in the setting of known parameters and unknown parameters, respectively {\cite{pmlr-v80-fazel18a}}. Based on this work, the convergence and sample complexity of PG method for several variants of LQR problem are analyzed \cite{doi:10.1137/20M1382386,9130755,9448427,10005813,9616447,zhang2021policy}. For finite-horizon LQR problem, Hambly et al. \cite{doi:10.1137/20M1382386} obtain the global linear convergence of PG method in the setting of both known and unknown parameters, and analyzes the convergence of projected PG method. Mohammadi et al. \cite{9448427} explore the sample complexity and convergence of PG method for infinite-horizon LQR problem with continuous-time system. Moreover, for LQR problem with risk constraints \cite{10005813}, LQR problem with decentralized system \cite{9616447}, $\mathcal{H}_2/\mathcal{H}_\infty$ problem \cite{zhang2021policy}, the convergence and sample complexity of PG method have also been established.

According to the duality between control and estimation, some  convergence results of PG method for solving KF problems have been obtained too  \cite{10156641,Talebi2023DatadrivenOF,Umenberger2022GloballyCP}. The work \cite{10156641} develops a receding-horizon policy gradient (RHPG) framework for learning the steady-state Kalman filter, and obtains the sample complexity of convergence of RHPG to optimal {filter with error $\epsilon$}. In  \cite{Talebi2023DatadrivenOF},  the problem of determining the steady-state Kalman gain is reformulated into a policy optimization problem for the adjoint system, and a SGD method is proposed to learn the steady-state Kalman gain with unknown noise covariance matrices and known model parameters. The work {\cite{Umenberger2022GloballyCP} solves the dynamic filter problem for output estimation by policy search method, and proposes a regularizer to establish the global convergence of policy search method.}
 
The objective of this paper is to solve the finite-horizon KF problem with unknown noise covariance matrices. Compared with existing results, our contributions are as follows. 
\begin{itemize}
	\item  We introduce a cost function on the observation and prove that the minimizer to this cost function is the Kalman gain; we further construct an unbiased estimation of gradient by using the model parameters and observations, and show the sample complexity and convergence rate of the resulting SGD method. 
	The results of sample complexity indicates that our estimation method needs less samples than that of zeroth-order optimization technique \cite{10156641}, which investigates a steady-state KF problem in the setting of both unknown model  parameters and unknown noise covariance matrices. 
	
	%By solving receding-horizon KF problems with SGD method and zeroth-order optimization technique, the  steady-state optimal solution to infinite-horizon KF problem is approximated and the sample complexity of this framework is also studied \cite{10156641}. 
	%
	%Compared with \cite{10156641}, we propose a SGD method to solve a  finite-horizon KF problem in the setting of unknown noise covariance matrices and yet known system parameters. 
	%
	%Inspired by \cite{10156641}, 

	\item The results of this paper can be viewed as an extension of \cite{Talebi2023DatadrivenOF} (learning the steady-state KF by a SGD method in the setting of known system parameters and unknown noise covariance matrices) to the setting of handling finite-horizon KF problem, and there are mainly two facts that make our results differ from those of \cite{Talebi2023DatadrivenOF}.
	Firstly, the KF is time-dependent in our setting, and we reformulate the learning problem as an optimal control problem of the dual system, which differs from the deterministic one for the backward adjoint system of \cite{Talebi2023DatadrivenOF}. %The duality of this paper is used to analyze the equivalence between the proposed optimal control problem and the finite-horizon KF problem, and yet the duality in \cite{Talebi2023DatadrivenOF} is used to propose the SGD method.
	Secondly, we consider the scenario that system noise and observation noise are both subgaussian; though explicit results are just given for the case of Gaussian noise, they can be easily extended to the setting of subgaussian noise by changing some coefficients only. Note that a bounded random variable is subgaussian and the work  \cite{Talebi2023DatadrivenOF} only considers the scenario of bonded noises. Interestingly, the non-asymptotic error bound of \cite{Talebi2023DatadrivenOF} depends explicitly on the bound of the noises.
	
	%The work \cite{Talebi2023DatadrivenOF} examines . Using the duality between control and estimation, the KF problem is  reformulated as an optimal control problem of the adjoint system, and a SGD method is introduced on this landscape; by the approach of high dimensional statistics, the non-asymptotic error bound guarantee of the proposed SGD method is presented under the assumption that the noise is bounded. 
	
	%\item We establish the global linear convergence guarantee of GD method for finite-horizon KF problem with known and unknown noise covariance matrices respectively, and provide the sample complexity in the setting of unknown noise covariance matrices. Compared with \cite{Talebi2023DatadrivenOF}, our Kalman gains is time-dependent with finite horizon; and unlike \cite{10156641}, we borrow the method from \cite {Talebi2023DatadrivenOF} to construct unbiased gradient estimation based on model parameters, instead of using the zeroth-order optimization method to estimate gradient.
\end{itemize}

The rest of this paper is organized as follows. In Section 2, we first review the standard KF, and then introduce some optimization problem which is shown to be equivalent to the finite-horizon KF problem. Section 3 presents the convergence analysis of exact  GD algorithm to our optimization problem. In Section 4, we present the form of gradient estimation and the framework of SGD method for solving the finite-horizon KF problem, whose convergence  analysis and sample complexity are also studied in this section. Section 5 is a specific discussion on our results and the related  works \cite{10156641,Talebi2023DatadrivenOF}. Section 6 presents some numerical examples to verify the feasibility of the proposed method.

{\textit{Notation}}. Let $\mathbb{N}$ be the set of all the natural numbers. For positive integers $n_0, n_1, n_2$, let $\mathcal{S}^{n_0}$ be the set of all the $\mathbb{R}^{n_0\times n_0}$-valued positive definite matrices, $I_{n_1}$ is the $n_1\times n_1$ unit matrix, $\mathbf{1}_{n_1\times n_1}$ denotes the $n_1\times n_2$ matrix whose elements are all $1$, and $0_{n_1\times n_2}$ is the $n_1\times n_2$ zero matrix. For a random variable $X$, $\mathbb{E}[X]$ is the mathematical expectation of $X$, $||X||_{L_2}$ is the $L_2$ norm of $X$, and $||X||_{\psi_2}$ is the $\psi_2$ norm of $X$. $\sum_{i=k}^{N}a_i$ denotes $a_N+ a_{N-1}+\cdots + a_k$ for  $k <  N$, $a_k$ for $k=N$ and $0$ for $k > N$. $\prod_{i=k}^{N}a_i$ denotes $a_N\times a_{N-1}\times\cdots \times a_k$ for  $k <  N$, $a_k$ for $k=N$ and $1$ for $k > N$. For matrix $A$, let $A^T$, $||A||$, $||A||_F$, $\sigma_{\min}^A$ and $\sigma_{\max}^A$ be, respectively, the transposition, spectral norm,  Frobenius norm, minimal and maximal singular values of $A$, and $A\otimes B$ denotes the kronecker product of matrices $A$ and $B$. If $A$ is square, $tr(A)$ denotes the trace of matrix $A$. 

\section{Problem formulation}

Our goal is to solve a finite-horizon KF problem with unknown noise covariance matrices, which is stated first in this section and then is transformed equivalently into an  optimization problem. Without loss of generality, we assume that the initial time $t_0$ is $0$ and the terminal time $t_{\max}$ is $M$, and denote $\mathbb{T}=\{0,\dots, M-1\}$.
\subsection{Kalman filter}\label{kf_intro}
%{\bf The setting. \/} 
Consider a discrete-time linear time-invariant system
\begin{equation}
	\left\{
	\begin{array}{ll}
		x_{t+1}=Ax_{t}+\omega_{t},\\
		y_{t}=Cx_{t}+v_{t}, ~~~
		t\in\mathbb{N},
	\end{array}
	\right.\label{eq0000}
\end{equation}
where $x_t\in\mathbb{R}^N$ and $y_t\in \mathbb{R}^m$ are the state vector and observation vector, respectively, and $x_0\sim \mathcal{N}(\overline{x}_0,\Sigma_{x_0})$. In (\ref{eq0000}), $\{v_t\}_{t=0}^\infty$ and $\{\omega_t\}_{t=0}^\infty$ are two sequences of Gaussian random vectors which denote the process noises and measurement noises, respectively, with properties 
\begin{align}
	\mathbb{E}[v_t]=0,\ \mathbb{E}[\omega_t]=0,\  \mathbb{E}[v_tv_t^T]=R_t,\ \mathbb{E}[\omega_t\omega_t^T]=Q_t.\label{eqn0}
\end{align} 
Moreover, $x_0, v_t, w_t, t\in\mathbb{N}$, are assumed to be  mutually independent.

As state $x_t$ in system \eqref{eq0000} cannot be obtained directly in practical problems, the goal of KF is to obtain the optimal estimation $\hat{x}_t$ of $x_t$ in the mean square sense based on the observation information $\{y_t\}$. Specifically, suppose that $\hat{x}_0$ has been known and $\hat{x}_0=\overline{x}_0$,  KF can be described as
\begin{equation}
	\left\{
	\begin{array}{ll}
		\widetilde{x}_{t+1}=A\hat{x}_t,\\
		\widetilde{y}_{t+1}=C\widetilde{x}_{t+1},\\
		\hat{x}_{t+1}=A\hat{x}_t+K_t^*(y_{t+1}-\widetilde{y}_{t+1}), ~~~ t\in\mathbb{N},
	\end{array}
	\right.\label{eqKF}
\end{equation}
where $\widetilde{x}_t$ is the prior estimation of $x_t$, and $K_t^*$ is the Kalman gain to minimize $\mathbb{E}[(x_{t+1}-\hat{x}_{t+1})^T(x_{t+1}-\hat{x}_{t+1})]$. In addition, KF can be rewritten as the following form without using prior estimation $\widetilde{x}_t$:
\begin{equation}\label{KFy}
	\begin{aligned}
\left\{
		\begin{array}{l}
			x_{t+1}=Ax_t+\omega_t,\\ 
			y_{t+1}=Cx_{t+1}+v_{t+1},\\
			\hat{x}_{t+1}=A\hat{x}_{t}+K_t^*CA(x_t-\hat{x}_t)+K_t^*C\omega_t+K_t^*v_{t+1},~~~
			t\in\mathbb{N}.
		\end{array}
		\right. 
	\end{aligned}
\end{equation}
\begin{assumption}\label{assumtion-1}
	Let covariance matrices $Q_t, R_{t}, t\in\mathbb{N}$, be positive definite.
\end{assumption}

The problem we want to solve is to learn Kalman gains $\{K_t^*\}_{t=0}^{M-1}$ in the setting of unknown noise covariance matrices, which we named the finite-horizon KF problem in what follows. $K_t^*,\ t\in\mathbb{T}$, can be calculated through the following equations
\begin{align}
	&P_{t+1}^*=AP_t^*A^T+Q_t-Z_t^*(H_t^*)^{-1}Z_t^{*T},\label{eq1}\\
	&K_{t}^*=Z_t^*(H_t^*)^{-1}\label{eq2}
\end{align}
with 
\begin{align*}
	&H_t^*=CAP_t^*(CA)^T+R_{t+1}+CQ_tC^T,\\
	&Z_t^*=AP_t^*(CA)^T+Q_tC^T.
\end{align*}
Here, the positive definiteness of $H_t^*$ is due to Assumption \ref{assumtion-1}. However, $Q_t, R_{t}, t\in\mathbb{T}$, and $P_0=\mathbb{E}[(x_0-\hat{x}_0)(x_0-\hat{x}_0)^T]$ are unknown in our setting, and solving (\ref{eq1}) (\ref{eq2}) is not feasible. %To solve this problem, we reformulate the learning problem into an optimization problem in Section \ref{reform_op}.

\subsection{Optimization formulation}\label{reform_op}

In this section, we reformulate the finite-horizon KF problem with unknown noise covariance matrices as an optimization problem, which is stated as the following one. 
%According to Problem 2, we construct an equivalent problem by observation $\{y_{i}\}_{i=2}^{M+N}$:

$\mathbf{Problem\ 1.}$ Find optimal $\mathbf{K}$ such that 
	\begin{equation}
		\begin{aligned}
			&	\min_{\mathbf{K}\in\mathbb{R}^{N\times Mm}} \mathbb{E}\Bigg[\sum_{t=0}^{M-1}\sum_{n=1}^{N}(y_{t+n+1}-\hat{y}_{t+1}^n)^T(y_{t+n+1}-\hat{y}_{t+1}^n)\Bigg]\\
			&~~~~\,\,s.t.~~~\left\{
			\begin{array}{l}
				x_{t+1}=Ax_t+\omega_t,\\
				y_{t+1}=Cx_{t+1}+v_{t+1},\\
				\hat{x}_{t+1}=A\hat{x}_{t}+K_tCA(x_t-\hat{x}_t)+K_tC\omega_t+K_tv_{t+1},
			\end{array}
			\right.	\label{pro4}
		\end{aligned}
	\end{equation}
	 where
	\begin{align}
		\hat{y}_{t}^n&:=CA^n\hat{x}_t=CA^nA^t\hat{x}_0+\sum_{i=0}^{t-1}CA^nA^{t-i-1}K_i(y_{i+1}-\hat{y}_i^1), n=1,\dots,N,\ t=1,\dots,M,\label{dey}\\
		\mathbf{K}&:=(K_0,\dots,K_{M-1}).
	\end{align}

  We will use the duality between control and estimation \cite{10.1115/1.3662552,402237} to prove the equivalence between Problem 1 and the finite-horizon KF problem. Note that the optimization function of Problem 1 can be rewritten as
\begin{align*}
	 tr&\Bigg(\mathbb{E}\big[\sum_{t=0}^{M-1}\sum_{n=1}^{N}(x_{t+1}-\hat{x}_{t+1})(x_{t+1}-\hat{x}_{t+1})^T\big(CA^n)^TCA^n\big]\Bigg)\\
	&+tr\Bigg(\sum_{t=0}^{M-1}\sum_{n=1}^{N}\sum_{i=0}^{n-1}Q_{t+n-i}(CA^i)^TCA^i+R_{t+n+1}\Bigg);
\end{align*} 
as the second part is a constant, Problem 1 is equal to the following problem.

$\mathbf{Problem\ 2.}$  Find optimal $\mathbf{K}$ such that 
	\begin{equation}
		\begin{aligned}
			&\min_{\mathbf{K}\in\mathbb{R}^{N\times Mm}} \mathbb{E} \Bigg[tr\Bigg(\sum_{t=0}^{M-1}{\sum_{n=1}^{N}  (x_{t+1}-\hat{x}_{t+1})(x_{t+1}-\hat{x}_{t+1})^T(CA^{n})^TCA^{n}}\Bigg)\Bigg]\\
			&~~~~\,\,\,s.t.~~~\left\{
			\begin{array}{l}
				x_{t+1}=Ax_t+\omega_t,\\
				y_{t+1}=Cx_{t+1}+v_{t+1},\\
				\hat{x}_{t+1}=A\hat{x}_{t}+K_tCA(x_t-\hat{x}_t)+K_tC\omega_t+K_tv_{t+1}.
			\end{array}
			\right.		\label{pro3}
		\end{aligned}
	\end{equation}

Let us introduce an optimal control problem of the dual system. 
 
$\mathbf{Problem\ 3.}$  Find optimal $\mathbf{K}$ such that 
\begin{align}
	\begin{aligned}
	~~~	&\min_{\mathbf{K}\in\mathbb{R}^{m\times MN}} \mathbb{E}\Bigg[ \sum_{t=0}^{M-1} {\bigg(s_t^TQ_{M-t-1}s_t+u_t^T(CQ_{M-t-1}C^T+R_{M-t})u_t{+}2u_t^TCQ_{M-t-1}s_t\bigg)} +{s_M^TP_0s_M}\Bigg] \\
		&~~~~~\,\,s.t.~~  \left\{
		\begin{array}{l}
			s_{t+1} = A^Ts_t+(CA)^Tu_t+z_t,\\
			u_t=-K_{M-t-1}^Ts_t,\\
		\end{array}
		\right.
	\end{aligned}\label{pro2}
\end{align}
with{
\begin{equation}\label{dual}
	s_0\sim \mathcal{N}(0,\Sigma), z_t\sim \mathcal{N}(0,
	\Sigma), t=0,\dots,M-2, z_{M-1}=0,
\end{equation}
}
where $s_0$ and $z_t, t\in\mathbb{T}$ are assumed to be mutually independent and {$\Sigma=\sum_{n=1}^{N}(CA^{n})^TCA^{n}$}.

\begin{remark}
	The dual system of Kalman filter is first introduced in  \cite{10.1115/1.3662552}, and we here make some changes in the deterministic one of \cite{10.1115/1.3662552} to adapt to system \eqref{eq0000} and Problem \eqref{pro3}, namely, the dual system in \eqref{pro2} is stochastic. Interestingly, the estimation problem can also be transformed into a LQR problem for some deterministic adjoint system  \cite{Talebi2023DatadrivenOF}, which is a backward equation in time.
\end{remark}

 The following theorem shows the equivalence between Problem 2 and Problem 3, namely, Problem 2 and Problem 3 can be transformed into a same optimization problem with respect to $\mathbf{K}$. Based on this theorem, the specific proof of equivalence between Problem 1 and the finite-horizon KF problem will be shown in Section \ref{eq_pro}.

 To simplify the subsequent proof and for any gains $\mathbf{K}=(K_0,\dots,K_{M-1})$, let  $A_t:=A-K_tCA, t\in\mathbb{T}$, and denote the optimization function of Problem 2 as 
\begin{equation}
\nonumber
f(\mathbf{K}):=\mathbb{E} \Bigg[tr\Bigg(\sum_{t=0}^{M-1}(x_{t+1}-\hat{x}_{t+1})(x_{t+1}-\hat{x}_{t+1})^T\Sigma\Bigg)\Bigg],
\end{equation}
where $x_t, t\in\mathbb{T}$ are generated by $\mathbf{K}$.

 \begin{theorem}\label{thdual}
{The optimal solution of Problem 2 is same to the one of Problem 3.}
 \end{theorem}
 
 \begin{proof}
 	The optimization function of Problem 3 can be written as
 	\begin{align}
 		&\mathbb{E}\Bigg[\bigg( \sum_{t=0}^{M-1} s_t^TQ_{M-t-1}s_t+u_t^T(CQ_{M-t-1}C^T+R_{M-t})u_t{+}2u_t^TCQ_{M-t-1}s_t\bigg) +s_M^TP_0s_M\Bigg]\notag\\
 		&=tr\Bigg(\sum_{t=0}^{M-1}\sum_{i=0}^{t}\prod_{j=M-t}^{M-i-1}A_{j}\left((I-K_{M-t-1}C)Q_{M-t-1}(I-K_{M-t-1}C)^T +K_{M-t-1}R_{M-t}K_{M-t-1}\right)\notag\\
 		&\hphantom{=}\cdot \bigg(\prod_{j=M-t}^{M-i-1}A_{j}\bigg)^T\Sigma\Bigg)
 		+tr\left(\sum_{i=0}^{M-1}\prod_{j=0}^{M-i-1}A_jP_0(\prod_{j=0}^{M-i-1}A_j)^T\Sigma\right). \label{eq5}
 	\end{align}
  According to \eqref{eq5}, we can rewritten Problem 3 as an optimization problem with respect to $\mathbf{K}$
  \begin{equation}
  	\begin{aligned}
  		\min_{\mathbf{K}\in\mathbb{R}^{m\times MN}}\ 
  		&\Bigg[tr\Bigg(\sum_{t=0}^{M-1}\sum_{i=0}^{t}\prod_{j=M-t}^{M-i-1}A_{j}
  		((I-K_{M-t-1}C)Q_{M-t-1}(I-K_{M-t-1}C)^T\\
  		&+K_{M-t-1}R_{M-t}K_{M-t-1})
  		\bigg(\prod_{j=M-t}^{M-i-1}A_{j}\bigg)^T\Sigma\Bigg)
  		+tr\Bigg(\sum_{i=0}^{M-1}\prod_{j=0}^{M-i-1}A_jP_0(\prod_{j=0}^{M-i-1}A_j)^T\Sigma\Bigg)\Bigg].  \label{op_pro}
  	\end{aligned}
  \end{equation}
Using \eqref{eqn0}, the optimization function in \eqref{op_pro} can be rewritten as
\begin{equation}
	\begin{aligned}
		&tr\Bigg(\sum_{t=0}^{M-1}\sum_{i=0}^{M-t-1}\prod_{j=t+1}^{M-i-1}A_{j}((I-K_{t}C)\mathbb{E}[\omega_t\omega_t^T](I-K_{t}C)^T+K_{t}\mathbb{E}[v_{t+1}v_{t+1}^T]K_{t})(\prod_{j=t+1}^{M-i-1}A_{j})^T\Sigma\Bigg)\\
		&+tr\Bigg(\sum_{i=0}^{M-1}\prod_{j=0}^{M-i-1}A_j\mathbb{E}[(x_0-\hat{x}_0)(x_0-\hat{x}_0)^T](\prod_{j=0}^{M-i-1}A_j)^T\Sigma\Bigg).\label{eq5_}
	\end{aligned}
\end{equation}
 	 According to \eqref{eqKF}, we have
 	\begin{align}
 		x_{t+1}-\hat{x}_{t+1}=\sum_{i=0}^{t}\prod_{j=i+1}^{t}A_j\big((I-K_iC)\omega_i-K_iv_{i+1}\big)+\prod_{j=0}^{t}A_j(x_0-\hat{x}_0),~~ t\in\mathbb{T}.\label{eq6}
 	\end{align}
 	Taking equation \eqref{eq6} into \eqref{eq5_}, we can obtain the optimization function of Problem 2
 	\begin{align}
 		\mathbb{E}\Bigg[tr\Bigg(\sum_{t=0}^{M-1}(x_{t+1}-\hat{x}_{t+1})(x_{t+1}-\hat{x}_{t+1})^T\Sigma\Bigg)\Bigg];
 	\end{align}
 this means that Problem 2 can also be transformed into  optimization problem \eqref{op_pro}. This completes the proof.
 \end{proof}

%\textbf{A simple implementation corresponding to Riccati equation \eqref{eq1} is}

%For the first problem, the problem \eqref{pro1} is not regular as a optimal problem which means it can have multiple optimal solutions. One method to solve this problem is adding noise into the equation of state \cite{doi:10.1137/20M1382386}, the finite optimal control problem can be solved by the policy gradient method for the Gaussian white noise $\omega_t$ make the gradient of Kalman gain $K_t$ becomes a combination of an inverse matrix and a matrix which implies the Riccati equation. For the filter problem we can use the arithmetic mean of $y_{t+1},\dots,y_0$ achieve the similar effect.

	To simplify the subsequent analysis, we propose the following symbolic
	\begin{equation}
	\left\{	\begin{aligned}
			&\mathcal{G}_t(X):=X+\sum_{i=0}^{M-t-1}\prod_{j=i}^{M-t-1}A_{M-j-1}^TXA_{M-j-1},\\
			&\mathcal{F}_t(X):=A_{M-t-1}^TXA_{M-t-1},\\
			&\mathcal{D}_{t,s}(X):=\prod_{i=s}^{M-t-1}A_{M-i-1}^TXA_{M-i-1}, t{,s}\in\mathbb{T},\label{defFGD}
		\end{aligned}\right.
	\end{equation}
	 where $\mathcal{G}_t(X)$ can be rewritten as
	\begin{align*}
		\mathcal{G}_t(X)=X+\sum_{i=0}^{M-t-1}\mathcal{D}_{t,i}(X).
	\end{align*}

\subsection{The equivalence between Problem 1 and the finite-horizon KF problem}\label{eq_pro}

To prove the equivalence between Problem 1 and the finite-horizon KF problem, we need to calculate $\nabla f(\mathbf{K})$. Let $\mathbf{K}$ be any gain, and define $P_t:=\mathbb{E}\big[(x_t-\hat{x}_t)(x_t-\hat{x}_t)^T\big],t=1,\dots,M$, where $\hat{x}_t,t\in\mathbb{T}$ are the states generated by gains $\mathbf{K}$. Using \eqref{pro3}, we obtain 
\begin{equation}
	\begin{aligned}
		&P_t=A_{t-1}P_{t-1}A_{t-1}^T+(I-K_{t-1}C)Q_{t-1}(I-K_{t-1}C)^T+K_{t-1}R_tK_{t-1}^T,\label{eqly}~~
		&t=1,\dots,M.
	\end{aligned}
\end{equation} 
Let $H_t$ and $Z_t$ be 
\begin{align*}
	&H_t:=CAP_t(CA)^T+R_{t+1}+CQ_tC^T,\\
	&Z_t:=AP_t(CA)^T+Q_tC^T,~~~~ t\in\mathbb{T}.
\end{align*}
%We get the following theorem to characterize $\nabla f(\mathbf{K})$ by the above equations.

\begin{theorem}\label{thg}
For any $\mathbf{K}\in\mathbb{R}^{m\times MN}$, the gradient $\nabla f(\mathbf{K}):=(\nabla_0 f(\mathbf{K}),\cdots, \nabla_{M-1} f(\mathbf{K}))$ has the form %can be represented as , where $\nabla_t f(\mathbf{K}), t\in\mathbb{T}$ are the gradient of $f(\mathbf{K})$ with respect to $K_t, t\in\mathbb{T}$, and
	\begin{equation}
		\nabla_{t} f(\mathbf{K})=2\Sigma_tE_t,~~ t\in\mathbb{T}, \label{eqg}
	\end{equation}
	where $\nabla_{t} f(\mathbf{K})$ is the gradient of $f(\mathbf{K})$ with respect to $K_t$ and 
	\begin{equation}
		\nonumber
		\begin{aligned}
			E_t&:=\big(K_tH_t-Z_t\big),\\
			\Sigma_t&:=\mathcal{G}_{t+1}(\Sigma).
		\end{aligned}
	\end{equation}
\end{theorem} 
\begin{proof}
	We first calculate $\nabla_0 f(\mathbf{K})$. According to Theorem \ref{thdual}, $f(\mathbf{K})$ is described as
	\begin{align}\label{f}
		f(\mathbf{K})=\mathbb{E}\Bigg[ \sum_{t=0}^{M-1} s_t^TQ_{M-t-1}s_t+u_t^T(CQ_{M-t-1}C^T+R_{M-t})u_t{+}2u_t^TCQ_{{M-t-1}}s_t +s_M^TP_0s_M\Bigg],
	\end{align}
	where $s_t,u_t, t\in\mathbb{T}$ are defined in Problem 3, and hence $f(\mathbf{K})$ can be transformed into
	\begin{align}
		f(\mathbf{K})
		=&\mathbb{E}\Bigg[\sum_{t=0}^{M-1} s_t^TQ_{M-t-1}s_t+u_t^T(CQ_{M-t-1}C^T+R_{M-t})u_t{+}2u_t^TCQ_{M-t-1}s_t\notag \\&+s_{M-1}^TA_0P_{0}A_0^Ts_{M-1}
		+z_{M-1}^TP_{0}z_{M-1}\Bigg]\notag\\
		=&\mathbb{E}\Bigg[\sum_{t=0}^{M-2} s_t^TQ_{M-t-1}s_t+u_t^T(CQ_{M-t-1}C^T+R_{M-t})u_t{+}2u_t^TCQ_{M-t-1}s_t\Bigg]\notag \\
		&+\mathbb{E}\Bigg[(s_{M-1}^TP_1s_{M-1})
		+z_{M-1}^TP_{0}z_{M-1}\Bigg]\label{P0}
	\end{align}
	with $z_t, t\in\mathbb{T}$ given in Problem 3. Note that the last equality of \eqref{P0} is by \eqref{eqly}.  As only $\mathbb{E}\big[s_{M-1}^TP_1s_{M-1}\big]$ contains $K_0$,  $\nabla_0 f(\mathbf{K})$ is represented as
	\begin{align*}
		\nabla_0 f(\mathbf{K})=\nabla_0 \mathbb{E}\big[s_{M-1}^TP_1s_{M-1}\big]
		=\nabla_0 tr\big(P_1\mathbb{E}\big[s_{M-1}s_{M-1}^T\big]\big),
	\end{align*}
	where 
	\begin{align}
		\mathbb{E}\big[s_{M-1}s_{M-1}^T\big]=&\mathbb{E}\Bigg[ \Bigg(\sum_{t=0}^{M-3}\prod_{i=t+1}^{M-2}A_{M-i-1}^Tz_t+z_{M-2}+\prod_{t=0}^{M-2}A_{M-t-1}^Ts_0\Bigg)\notag\\
		&\Bigg(\sum_{t=0}^{M-3}\prod_{i=t+1}^{M-2}A_{M-i-1}^Tz_t+z_{M-2}+\prod_{t=0}^{M-2}A_{M-t-1}^Ts_0\Bigg)^T\Bigg]\notag\\
		=&\mathbb{E}\Bigg[\sum_{t=0}^{M-3}\mathcal{D}_{1,t+1}(z_tz_t^T)+z_{M-2}z_{M-2}^T+\mathcal{D}_{1,0}(s_0s_0^T)\Bigg]\notag\\
		=&\mathcal{G}_1(\Sigma)\notag\\
		=&\Sigma_0. \label{eq7}
	\end{align}
	As $\Sigma_0$ has nothing to do with $K_0$, substituting $P_1$ into $P_t$ of \eqref{eqly} leads to 
		\begin{align*}
		\nabla_0 f(\mathbf{K}) =& 2\Sigma_0(K_0R_1-(I-K_0C)Q_0C^T-A_0P_0(CA)^T)\\
		=&2\Sigma_0E_0.
	\end{align*}
	According to \eqref{eqly}, $f(\mathbf{K})$ can also be rewritten as
	\begin{equation}\label{f(K)expansiont}
		\begin{aligned}
			f(\mathbf{K})=&\mathbb{E}\Bigg[\sum_{i=0}^{M-t-1} s_i^TQ_{M-i-1}s_i+u_i^T(CQ_{M-i-1}C^T+R_{M-i})u_i{+}2u_i^TCQ_{M-i-1}s_i \\
			&+s_{M-t}^TP_{t}s_{M-t}+\sum_{i=M-t}^{M-1}z_iP_{M-i-1}z_i\Bigg],~~~
			t=2,\dots,M-1,\\
		\end{aligned}
	\end{equation}
	and
	\begin{equation}
		f(\mathbf{K})=\mathbb{E}\Bigg[s_{0}^TP_{M}s_{0}+\sum_{i=0}^{M-1}z_iP_{M-i-1}z_i\Bigg].\label{f(k)expansion}
	\end{equation}
	In \eqref{f(K)expansiont} and \eqref{f(k)expansion}, only $s_{M-t}^TP_{t}s_{M-t}$ contains $K_t$ and we have 
	\begin{align*}
		\nabla_t f(\mathbf{K})=&\nabla_t \mathbb{E}[s_{M-t-1}^TP_{t+1}s_{M-t-1}]\\
		=&2\Sigma_tE_t,{	~~~t\in\mathbb{T}}.
	\end{align*}
	in a similar way.
\end{proof}

Using \eqref{f(k)expansion}, we get the upper bound of $||P_t||$
\begin{equation}\label{ieqp}
	||P_t||\leq\sum_{i=0}^{M}||P_i||\leq \frac{f(\mathbf{K})}{\sigma_{\min}^{\Sigma}}+||P_0||,~~ t\in\mathbb{T}. 
\end{equation}
 According to the definition of $s_{t},t=0,\dots,M$, we have
\begin{equation}\label{ieqsigma_}
	\begin{aligned}
		f(\mathbf{K})=&\mathbb{E}\Bigg[\sum_{i=0}^{M-t-2} s_i^TQ_{M-i-1}s_i+u_i^T(CQ_{M-i-1}C^T+R_{M-i})u_i\textcolor{red}{+}2u_i^TCQ_{M-i-1}s_i \\
		&+u_{M-t-1}^TR_{t+1}u_{M-t-1}+(s_{M-t}-z_{M-t-1})^TA^{-T}Q_tA^{-1}(s_{M-t}-z_{M-t-1})\\
		&+s_{M-t}^TP_{t}s_{M-t}+\sum_{i=M-t}^{M-1}z_iP_{M-i-1}z_i\Bigg]
	\end{aligned}
\end{equation} 
with $t\in\mathbb{T}.$
We obtain the upper bound of $\Sigma_t, t\in\mathbb{T}$ by \eqref{ieqsigma_}
\begin{equation}\label{bdsigma}
	||\Sigma_t||\leq tr(\Sigma_t)\leq \frac{f(\mathbf{K})}{\sigma_{\min}^{\mathbf{A^{-T}QA^{-1}}}}+N\sigma_{\max}^{\Sigma},
\end{equation}
where $\sigma_{\min}^{\mathbf{A^{-T}QA^{-1}}}$ denotes $\min_{t\in\mathbb{T}}\{\sigma_{\min}^{A^{-T}Q_tA^{-1}}\}$. Introduce the notations
$\mathcal{A}(\mathbf{K}):=\frac{f(\mathbf{K})}{\sigma_{\min}^{\Sigma}}+||P_0||$ and $\mathcal{B}(\mathbf{K}):=\frac{f(\mathbf{K})}{\sigma_{\min}^{\mathbf{A^{-T}QA^{-1}}}}+N\sigma_{\max}^{\Sigma}$ to simplify the subsequent proof.

By Theorem \ref{thg} and the following assumption about system \eqref{eq0000}, we can prove that the optimal solution of Problem 2 satisfies equations \eqref{eq1} and \eqref{eq2}; this means the optimal solution of Problem 2 is equal to Kalman gain $\{K_t^*\}_{t=0}^{M-1}$. 
\begin{assumption}\label{as1}
 The pair $(C,A)$ is observable.
\end{assumption}

\begin{assumption}\label{as3}

Matrix $A$ is invertible.
\end{assumption}

\begin{corollary}\label{eq_op}
	Under assumptions \ref{assumtion-1}-\ref{as3}, the optimal solution of Problem 2 is equal to $\{K_t^*\}_{t=0}^{M-1}$.
\end{corollary}
\begin{proof}
Let $\mathbf{K}^\circ$ be the optimal solution of Problem 2. Substituting \eqref{eqg} into $\nabla_t f(\mathbf{K}^\circ)=0,\ t\in\mathbb{T}$, we have
	\begin{align*}
		2\Sigma_t(K_tH_t-Z_t)=0.
	\end{align*}
	Matrix $\Sigma_t$ is invertible under Assumption \ref{as1} and \ref{as3}. Therefore, $E_t=K_tH_t-Z_t=0$. Due to Assumption \ref{assumtion-1}, we have
	\begin{align}
		K_t=Z_tH_t^{-1}.\label{K}
	\end{align}
	Taking \eqref{K} into \eqref{eqly}, $P_t,\ t\in\mathbb{T}$ satisfy
	\begin{align*}
		P_{t+1}=&A_{t}P_{t}A_{t}^T+(I-K_{t}C)Q_{t}(I-K_{t}C)^T+K_{t}R_{t+1}K_{t}^T\notag\\
		=&AP_{t}A^T+Q_{t}-2Z_{t}H_{t}^{-1}Z_{t}+Z_{t}H_{t}^{-1}H_{t}H_{t}^{-1}Z_{t}^T\notag\\
		=&AP_{t}A^T+Q_{t}-Z_{t}H_{t}^{-1}Z_{t}^T,
	\end{align*}
which means $P_t=P_t^*, t=0,\dots,M$. That completes the proof.
\end{proof}

Since Problem 1 is equal to Problem 2, the equivalence between Problem 1 and the finite-horizon KF problem is proved by Corollary \ref{eq_op}.
For Problem 1, we propose a SGD method in Section \ref{sgdconv} to solve it. To analyze the convergence of SGD method, we first analyze the global linear convergence guarantee of exact GD method in Section \ref{eq_conv}.

\subsection{Discussion on Assumption 3}

Assumption \ref{as3} might be too strong, and it means that the analysis of our method can be applied to the case that the system matrix is invertible. In fact, our initial idea is to solve the following optimization problem instead of Problem 1.
	
	$\mathbf{Problem\ 4.}$ Find optimal $\mathbf{K}$ such that 
	\begin{equation}
		\begin{aligned}
			&\min_{\mathbf{K}\in\mathbb{R}^{m\times MN}} \mathbb{E} \Bigg[tr\Bigg(\sum_{t=0}^{M-1}  (x_{t+1}-\hat{x}_{t+1})(x_{t+1}-\hat{x}_{t+1})^T\Sigma_{initial}\Bigg)\Bigg]\\
			&~~~~\,\,\,s.t.~~~\left\{
			\begin{array}{l}
				x_{t+1}=Ax_t+\omega_t,\\
				y_{t+1}=Cx_{t+1}+v_{t+1},\\
				\hat{x}_{t+1}=A\hat{x}_{t}+K_tCA(x_t-\hat{x}_t)+K_tC\omega_t+K_tv_{t+1}, 
			\end{array}
			\right.		\label{pro3_init}
		\end{aligned}
	\end{equation}
where $\Sigma_{initial}=\sum_{n=0}^{N-1}(CA^{n})^TCA^{n}$. 

Note that there is only a little difference between $\Sigma$ and $\Sigma_{initial}$, and hence the analysis of Problem 2, such as Theorem \ref{thg}, can be easily extended to the setting for  Problem 4. However, it is not easy to rewrite the optimization function of \eqref{pro3_init} by the observations of system \eqref{eq0000} indeed. Specifically, optimization function of Problem 4 is the sum of 
\begin{equation}
	tr\big(\mathbb{E}[(CA^nx_{t+1}-CA^n\hat{x}_{t+1}) (CA^nx_{t+1}-CA^n\hat{x}_{t+1})^T]\big), ~~t\in\mathbb{T}, ~~n=0,\dots,N-1.
\end{equation}
 Using equations \eqref{KFy} and \eqref{dey}, we have
\begin{equation}
	\begin{aligned}
		&\mathbb{E}[(y_{t+n+1}-\hat{y}_{t+1}^n)^T(y_{t+n+1}-\hat{y}_{t+1}^n)]\\
		&=tr\big(\mathbb{E}[(x_{t+1}
		-\hat{x}_{t+1})(x_{t+1}-\hat{x}_{t+1})^T((CA^n)^TCA^n)]\\
		&\hphantom{=}-2\mathbb{E}[v_{t+n+1}\hat{x}_{t+1}^T]+\mathbb{E}[v_{t+n+1}v_{t+n+1}^T]\big),~~~t\in\mathbb{T},n=1,\dots,N-1.
	\end{aligned}
\end{equation}
As $v_{t+n+1}$ is linear independent with $\hat{x}_{t+1}$, we have that $\mathbb{E}[v_{t+n+1}\hat{x}_{t+1}]=0$ and  $\mathbb{E}[v_{t+n+1}v_{t+n+1}^T]$ is a constant. Therefore, the optimal solution of Problem 4 does not change if we  replace $tr\big(\mathbb{E}[(CA^nx_{t+1}-CA^n\hat{x}_{t+1})$ $(CA^nx_{t+1}-CA^n\hat{x}_{t+1})^T]\big), t\in\mathbb{T}, n=1,\dots,N-1$ by $tr\big(\mathbb{E}[(y_{t+n+1}-\hat{y}_{t+n+1}^n)^T(y_{t+n+1}-\hat{y}_{t+1}^n)]\big), t\in\mathbb{T}, n=1,\dots,N-1$. On the other hand,   
\begin{equation}
	\begin{aligned}
		\mathbb{E}[(y_{t+1}-\hat{y}_{t+1}^0)^T(y_{t+1}-\hat{y}_{t+1}^0)]=&tr\big(\mathbb{E}[(x_{t+1}
		-\hat{x}_{t+1})(x_{t+1}-\hat{x}_{t+1})^T(C^TC)]\\
		&-2\mathbb{E}[\hat{x}_{t+1}v_{t+1}^T]+\mathbb{E}[v_{t+1}v_{t+1}^T]\big),
	\end{aligned}
\end{equation}
and $\mathbb{E}[\hat{x}_{t+1}v_{t+1}^T]=K_tR_{t+1}$ is not equal to zero; if we further replace $tr\big(\mathbb{E}[(x_{t+1}
-\hat{x}_{t+1})(x_{t+1}-\hat{x}_{t+1})^T(C^TC)]\big)$ by $tr\big(\mathbb{E}[(y_{t+1}-\hat{y}_{t+1}^0)^T(y_{t+1}-\hat{y}_{t+1}^0)]\big)$, the superfluous term $K_tR_{t+1}$ will change the optimal solution of Problem 4, which will be no longer the Kalman gain $\{K_t^*\}_{t=0}^{M-1}$.
	
Interestingly, we do not need Assumption \ref{as3} in the following special condition. If $R_{t}\equiv R, t\in\mathbb{N}$ and $P_0=0$ (i.e., $x_0=\hat{x}_0$), then we can have the estimation $\hat{R}=\frac{1}{L}\sum_{l=1}^{L}v_0(l)v_0^T(l)$ of covariance matrix $R$ of $v_t, t\in\mathbb{N}$, by collecting the samples of measurement noise $v_0(l)=y_0(l)-Cx_0, l=1,\dots,L$.
Therefore, we can replace $tr\big(\mathbb{E}[(x_{t+1}
-\hat{x}_{t+1})(x_{t+1}-\hat{x}_{t+1})^T(C^TC)]\big)$ in Problem 4 by $tr(\mathbb{E}[(y_{t+1}-\hat{y}_{t+1}^0)^T(y_{t+1}-\hat{y}_{t+1}^0)]+2K_t\hat{R})$ without changing the optimal solution of Problem 4, and hence we can propose a learning method like Algorithm 1 below for Problem 4 and analyze the sample complexity and convergence of the new algorithm as the same as we will do in Section \ref{sgdconv}.

Indeed, Assumption \ref{as3} has been proposed in the classical Mehra method \cite{1100100}, which is used to estimate the noise covariance. The necessity of this condition has yet to be proved \cite{ODELSON2006303}, and whether the Kalman gain can be obtained in the setting with singular $A$ has  been not addressed. In the near future, we might explore how to weaken this assumption in the perspective of optimization.

\section{Exact GD method}\label{eq_conv}

{In this section, we first review some properties of $f(\mathbf{K})$ and analyze the convergence of exact GD method. Though the model of our optimization problem is different from the ones of existing literature, the proofs of this section are essentially same to those of \cite{doi:10.1137/20M1382386} due to the relationship between Kalman filter and Kalman predictor. 
Just for the completeness, we present some concise proofs of the concerned results.}

Suppose that covariance matrices $P_0$, $\{Q_t\}_{t\in\mathbb{T}},$ $\{R_{t+1}\}_{t\in\mathbb{T}}$ and  model parameters $A, C$ are all known, and in this section we will analyze the convergence of exact GD method for Problem 2. To simplify the subsequent proofs, we introduce the following notations. Let $\mathbf{K}'=(K_0',K_1',\dots,K_{M-1}')$ be any gain and $\mathbf{K}^*=(K_0^*,K_1^*,\dots,K_{M-1}^*)$ be the optimal solution of Problem 2, and $A_t',P_{t+1}', \Sigma_t', E_t', A_t^*, \Sigma_t^*, E_t^*, t\in\mathbb{T}$ are defined in the same way as $A_t, P_{t+1}, \Sigma_t, E_t$, where we just replace $K_t$ by $K_t'$ and $K_t^*$, respectively. Let 
$\{s_t'\}$ and $\{u_t'\}$ denote the state sequence and control sequence generated by policy $\mathbf{K}'$, and further let  $\sigma_{\max}^\mathbf{CQC^T+R}$, $\sigma_{\min}^\mathbf{R}$ denote  $\max_{t\in\mathbb{T}}\{\sigma_{\max}^{CQ_tC^T+R_{t+1}}\}$ and  $\min_{t\in\mathbb{T}}\{\sigma_{\min}^{R_{t+1}}\}$, respectively.  Denote
\begin{equation}
	\begin{aligned}
		&\Lambda_{t}:=(CQ_{t}C^T)+R_{t+1}+CAP_{t}(CA)^T,\\
		&\delta K_{t}':=K_{t}'-K_{t},~~~t\in\mathbb{T},\\
		&||\Lambda_{\max}||:=\sigma_{\max}^\mathbf{CQC^T+R}+||CA||^2\mathcal{A}(\mathbf{K}).\label{ieLambda}
	\end{aligned}
\end{equation}
 
%\subsection{The convexity and smoothness of $f(\mathbf{K})$}

The Polyak-Lojasiewicz (PL) condition is key to analyze the convergence of exact GD method, and it can be deduced in several  different versions of LQR problem  \cite{pmlr-v80-fazel18a,doi:10.1137/20M1382386,9130755,10156641,Talebi2023DatadrivenOF}. Lemma \ref{lpl} and Lemma \ref{lsmooth} below are simple extensions of Lemma 3.6 and Lemma 3.7 of  \cite{doi:10.1137/20M1382386}, which shows that $f(\mathbf{K})$ also has the PL condition and almost smoothness. Note that the expression (\ref{f}) of $f(\mathbf{K})$ contains a cross term between $u_t$ and $s_t$; this makes the coefficients $c_1, c_2$ of (\ref{ieq00}) (\ref{ieq000}) slightly different from those in \cite{doi:10.1137/20M1382386}. For the completeness, this paper presents a concise proof of Lemma \ref{lpl}, which is same to that of Lemma 3.6 of \cite{doi:10.1137/20M1382386}.

\begin{lemma}\label{lpl}
	Under assumptions \ref{assumtion-1}-\ref{as3} and for any  $\mathbf{K}\in\mathbb{R}^{N\times mM}$, $f(\mathbf{K})$ satisfies the following inequalities:
	\begin{align}
		f(\mathbf{K})-f(\mathbf{K}^*)&\leq c_1tr\Bigg(\sum_{t=0}^{M-1}\nabla_t f(\mathbf{K})\nabla_t f(\mathbf{K})^T\Bigg)\label{ieq00}, \\
		f(\mathbf{K})-f(\mathbf{K}^*)&\geq c_2tr\Bigg(\sum_{t=0}^{M-1}E_tE_t^T\Bigg),\label{ieq000}
	\end{align}
where $c_1=\frac{\mathcal{B}(\mathbf{K}^*)}{4\sigma_{\min}^R(\sigma_{\min}^\Sigma)^2}$ and $c_2=\frac{\sigma_{\min}^\Sigma}{4(\sigma_{\max}^\mathbf{CQC^T+R}+||CA||^2\mathcal{A}(\mathbf{K}))}$.
\end{lemma}
\begin{proof} The advantage function takes the form 
	\begin{align}
			A(\mathbf{K}',\mathbf{K})=&f(\mathbf{K}')-f(\mathbf{K})\notag\\
			%=&\mathbb{E}\Bigg[\sum_{t=0}^{M-1} s_t'^TQ_{M-t-1}s_t'+u_t'^T(CQ_{M-t-1}C^T+R_{M-t})u_t'-2u_t'^TCQ_{M-t-1}s_t' \notag\\
			%&+s_t'P_{M-t}s_t'
			%-s_t'P_{M-t}s_t'
			%+s_M'^TP_0s_M'\Bigg]-\mathbb{E}\Bigg[s_0^TP_{M}s_0+\sum_{t=0}^{M-1}z_tP_{M-t-1}z_t^T\Bigg]\notag\\
			=&\mathbb{E}\Bigg[\sum_{t=0}^{M-1} s_t'^TQ_{M-t-1}s_t'+u_t'^T(CQ_{M-t-1}C^T+R_{M-t})u_t'{+}2u_t'^TCQ_{M-t-1}s_t' \notag\\
			&+s_{t+1}'P_{M-t-1}s_{t+1}'
			-s_t'P_{M-t}s_t'-z_tP_{M-t-1}z_t^T\Bigg]\notag\\
%			=& tr\Bigg(2\sum_{t=0}^{M-1}\delta K_{M-t-1}'^T\Sigma_{M-t-1}' E_{M-t-1}
%			+\sum_{t=0}^{M-1} \Sigma_{M-t-1}'(\delta K_{M-t-1}'\Lambda_{M-t-1}\delta K_{M-t-1}'^T)\Bigg)\notag\\
%			%=& tr\Bigg(2\sum_{t=0}^{M-1}\delta K_{t}'^T\Sigma_{t}' E_{t}
%			%+\sum_{t=0}^{M-1} \Sigma_{t}'(\delta K_{t}'\Lambda_{t}\delta K_{t}'^T)\Bigg)\notag\\
			=& tr\Bigg(\sum_{t=0}^{M-1}\Sigma_{t}'(E_{t}\Lambda_{t}^{-1}+\delta K_{t}')\Lambda_{t}(E_{t}\Lambda_{t}^{-1}+\delta K_{t}')^T
			-\Sigma_{t}'E_{t}\Lambda_{t}^{-1}(E_{t})^T\Bigg). 	\label{eq8}	
	\end{align}
%	Using the same method in \cite{doi:10.1137/20M1382386,9130755} to completes the rest proof. 
Substituting ${\mathbf{K}' = \mathbf{K}^*}$ into \eqref{eq8}, we obtain
	\begin{align}
		A(\mathbf{K}^*,\mathbf{K})
%		=& tr\Bigg(\sum_{t=0}^{M-1}\Sigma_{t}^*(E_{t}\Lambda_{t}^{-1}+\delta K_{t}^*)\Lambda_{t}(E_{t}\Lambda_{t}^{-1}+\delta K_{t}^*)^T
%		-\Sigma_{t}^*E_{t}\Lambda_{t}^{-1}(E_{t})^T\Bigg)\notag\\
		&\geq -tr\Bigg(\sum_{t=0}^{M-1}\Sigma_{t}^*E_{t}\Lambda_{t}^{-1}(E_{t})^T\Bigg).\label{ieq0}
	\end{align}
	Dividing both sides of \eqref{ieq0} by -1 and using inequalities $||\Lambda_{t}^{-1}||\leq \frac{1}{\sigma_{\min}^\mathbf{R}}$, $||\Sigma_{t}^{-1}||\leq \frac{1}{\sigma_{\min}^\Sigma}$, we have 
	\begin{align*}
		f(\mathbf{K})-f(\mathbf{K}^*)\leq \sum_{t=0}^{M-1}\frac{||\Sigma_{t}^*||}{4\sigma_{\min}^\mathbf{R}(\sigma_{\min}^\Sigma)^2}||\nabla_tf(\mathbf{K})||_F^2.%\label{pl}
	\end{align*}
According to \eqref{bdsigma}, $c_1=\frac{\mathcal{B}(\mathbf{K}^*)}{4\sigma_{\min}^R(\sigma_{\min}^\Sigma)^2}\geq\frac{||\Sigma_{t}^*||}{4\sigma_{\min}^\mathbf{R}(\sigma_{\min}^\Sigma)^2}$, and hence we get \eqref{ieq00}.
	Similarly, substituting $\hat{K}_t=K_t-E_t\Lambda_t^{-1}$ into $K_t'$ of \eqref{eq8}, it holds 
	\begin{align}
		f(\mathbf{K})-f(\mathbf{K}^*)\geq f(\mathbf{K})-f(\mathbf{\hat{K}})
		%=& tr\Bigg(\sum_{t=0}^{M-1}\hat{\Sigma}_{t}E_{t}\Lambda_{t}^{-1}(E_{t})^T\Bigg)\notag\\
		\geq \sum_{t=0}^{M-1}\frac{\sigma_{\min}^\Sigma}{4||\Lambda_{\max}||}||E_t||_F^2\geq\sum_{t=0}^{M-1}c_2||E_t||_F^2.\label{lowerb}
	\end{align} 
	This completes the proof. 
\end{proof}
	Using Lemma \ref{lpl}, $\sum_{t=0}^{M-1}||K_t||$ can be bounded in the following way
	\begin{align}
			\sum_{t=0}^{M-1}||K_t||&\leq \sum_{t=0}^{M-1}\Bigg( \frac{||(K_tH_t-Z_t)||+||Z_t||}{\sigma_{\min}^\mathbf{R}}\Bigg)\notag\\
			%&\leq \sum_{t=0}^{M-1}\Bigg(\frac{||E_t||}{\sigma_{\min}^\mathbf{R}}+\frac{||Z_t||}{\sigma_{\min}^\mathbf{R}}\Bigg)\notag\\
%			&\leq  \frac{\sqrt{Mc_2^{-1}(f(\mathbf{K})-f(\mathbf{K}^*))}}{\sigma_{\min}^\mathbf{R}}+\sum_{t=0}^{M-1}\frac{||Z_t||}{\sigma_{\min}^\mathbf{R}}\notag\\
			&\leq
			\frac{\sqrt{Mc_2^{-1}(f(\mathbf{K})-f(\mathbf{K}^*))}+\mathcal{A}(\mathbf{K})||CA||||A||}{\sigma_{\min}^\mathbf{R}}+\sum_{t=0}^{M-1}\frac{||CQ_t||}{\sigma_{\min}^\mathbf{R}}\notag\\
			&:=\mathcal{C}(\mathbf{K}). \label{ieqK}
	\end{align}
%	The third inequality holds by \eqref{ieq00} and $\sum_{t=0}^{M-1}||E_t||\leq \sqrt{M\sum_{t=0}^{M-1}||E_t||_F^2}$, and the forth inequality holds by \eqref{ieqp}. To simplify the subsequent proof, 
%We define $\mathcal{C}(\mathbf{K}):=	\frac{\sqrt{Mc_2^{-1}(f(\mathbf{K})-f(\mathbf{K}^*))}+\mathcal{A}(\mathbf{K})||CA||||A||}{\sigma_{\min}^\mathbf{R}}+\sum_{t=0}^{M-1}\frac{||CQ_t||}{\sigma_{\min}^\mathbf{R}}$.
	%
 According to \eqref{eq8}, we can get the almost smoothness of $f(\mathbf{K}) $.
\begin{lemma}\label{lsmooth}
	%$\forall\ \mathbf{K},\mathbf{K}'\in\mathbb{R}^{m\times MN},$% 
	$f(\mathbf{K})$ and $f(\mathbf{K}')$ satisfies
	\begin{equation}
		f(\mathbf{K}')-f(\mathbf{K})=tr\Bigg(2\sum_{t=0}^{M-1}\Sigma_{t}' E_{t}\delta K_{t}'^T
		+\sum_{t=0}^{M-1} \Sigma_{t}'\delta K_{t}'\Lambda_{t}\delta K_{t}'^T\Bigg).
	\end{equation}
\end{lemma} 

%\subsection{The Local Lipschitz continuity of $\Sigma_t$}

    Denote $\rho:= \max_{t\in\mathbb{T}}\{\max\{||A_t||+\frac{1}{2}, 1+b\}\}$, $|||\mathbf{K}|||=\sum_{t=0}^{M-1}||K_t||$ and $|||\mathbf{K}-\mathbf{K}'|||=\sum_{t=0}^{M-1}||K_t-K_t'||$, respectively, where $b>0$ is a small constant.
    We now propose the following lemma to characterize the local Lipschitz continuity of $\Sigma_t$,
 
 \begin{lemma}\label{lePSi}
 	Under Assumption \ref{assumtion-1} and condition $||K_{t}-K_{t}'||\leq\frac{1}{2||CA||}, t\in\mathbb{T}$, it holds 
 	\begin{equation}
 		\begin{aligned}
 		     &||\Sigma_t-\Sigma_t'||
 			\leq \big(\frac{\rho^{2M}-1}{\rho^2-1}(2\rho+1)\big)||\Sigma||||CA|||||\mathbf{K}-\mathbf{K}'|||,~~~ t\in\mathbb{T}.\label{ieq3/4}
 		\end{aligned}
 	\end{equation}
 \end{lemma}

\begin{proof}
	
By Corollary 3.14 of \cite{doi:10.1137/20M1382386}, we get
\begin{equation}
	\sum_{i=0}^{M-1}||\mathcal{D}_{t,i}(\Sigma)-\mathcal{D}_{t,i}'(\Sigma)||\leq \sum_{i=0}^{M-1}\big(\frac{\rho^{2M}-1}{\rho^2-1}\big) ||\mathcal{F}_{M-i-1}(\Sigma)-\mathcal{F}_{M-i-1}'(\Sigma) ||.\label{LLG}
\end{equation}
As $||K_{t}-K_{t}'||\leq\frac{1}{2||CA||}$, it is easy to have  $||A_t'||\leq||A_t||+\frac{1}{2}$. 
According to the definition of $\mathcal{F}_t$ in \eqref{defFGD},  we can completes the proof by using Lemma 3.12 of  \cite{doi:10.1137/20M1382386}.
\end{proof}
%By $\rho= \max_{t\in\mathbb{T}}\{\max\{||A_t||+\frac{1}{2}, 1+b\}\}$ and the definition of $\{\Sigma_t\}_{t=0}^{M-1}$, $\{\Sigma_t\}_{t=0}^{M-1}$ can be bounded as
%\begin{equation}
%||\Sigma_t||\leq \frac{\rho^{2M}-1}{\rho^2-1}||\Sigma||.\label{bdsigma}
%\end{equation}
%\subsection{Convergence of GD method}

We consider to analyze the exact GD method with updating rule
\begin{equation}
	\nonumber
	\mathbf{K}_{k+1}=\mathbf{K}_k-\eta\nabla f(\mathbf{K}_k),
\end{equation}
where $\mathbf{K}_k=(K_0^k,\dots,K_{M-1}^k)$ is the $k$-th iteration point and $\eta$ is a fixed step size. Many works have analyzed the convergence rate of exact GD method for LQR problem   \cite{pmlr-v80-fazel18a,doi:10.1137/20M1382386,9130755}, and we can use roughly the same method to get the similar result. To simplify the proof, we define $P_{t+1}^k, A_t^k,\Sigma_t^k,\Lambda_t^k,E_t^k,t\in\mathbb{T}$ in the same way as $P_{t+1},A_t,\Sigma_t,\Lambda_t,E_t$ above.

\begin{theorem}\label{th1}
	Let assumptions \ref{assumtion-1}-\ref{as3} hold and the step size $\eta$ satisfy  
	\begin{equation}
		\begin{aligned}
			&\eta \leq\min\{c_3,c_4\},\\
			&c_3=\frac{\min\{1,\sigma_{\min}^\Sigma\}}{2\mathcal{B}(\mathbf{K}_0)\sqrt{Mc_2^{-1}\big(\{f(\mathbf{K}_0)-f(\mathbf{K}^*)\}\big)}\big((\frac{\rho_{\max}^{2M}-1}{\rho_{\max}^2-1})(2\rho_{\max}+1)\big)||\Sigma||\max\{||CA||,1\}},\\
			&c_4=\frac{1}{4(\sigma_{\max}^\mathbf{CQC^T+R}+||CA||^2\mathcal{A}(\mathbf{K}_0))(\frac{1}{2}+\mathcal{B}(\mathbf{K}_0))}
		\end{aligned}\label{eta0}
	\end{equation}   
	with $\rho_{\max}:= \max_{t\in\mathbb{T}}\{||A||+||CA||\mathcal{C}({\mathbf{K}}_0), 1+b\}$ and $b>0$ a small constant. Then,
	\begin{equation}\label{l1}
		\begin{aligned}
					f(\mathbf{K}_{k+1})-f(\mathbf{K}^*)\leq(1-2\alpha)^{k+1} (f(\mathbf{K}_0)-f(\mathbf{K}^*)),\ 
					\alpha = \frac{\eta}{c_18},\ 
					k\in\mathbb{N},
		\end{aligned}
	\end{equation}
	where  $\{\mathbf{K}_k\}$ is the filter gain sequence generated by the exact GD method. 
	
\end{theorem}
\begin{proof}
  Consider the sequence $\{\hat{\mathbf{K}}_{k}\}$ generated by $\hat{\mathbf{K}}_{k+1}=\hat{\mathbf{K}}_{k}-\eta_k\nabla f(\hat{\mathbf{K}}_{k}), k\in\mathbb{N},$ with  $\hat{\mathbf{K}}_{0}=\mathbf{K}_0$, and let $\hat{P}_{t+1}^k, \hat{A}_{t}^k,\hat{\Sigma}_{t}^k,\hat{\Lambda}_{t}^k,\hat{E}_{t}^k,t\in\mathbb{T}$ be defined in the same way as that of  ${P}_{t+1},{A}_{t},{\Sigma}_{t},{\Lambda}_{t},{E}_{t}$. Set
  	$\eta_k \leq\min\{c_3^k,c_4^k\}$ with $c_3^k, c_4^k$ similarly defined as $c_3, c_4$ by replacing $\mathbf{K}_0$ and $\rho_{\max}$ with $\hat{\mathbf{K}}_k$ and $\rho_k:= \max_{t\in\mathbb{T}}\{\max\{||\hat{A}_t^k||+\frac{1}{2}, 1+b\}\}$, respectively.  Using Lemma \ref{lsmooth}, we have
	\begin{equation}
		f(\hat{\mathbf{K}}_{k+1})-f(\hat{\mathbf{K}}_k)
%=&tr\Bigg(-2\eta_k\sum_{t=0}^{M-1}\nabla_t f(\hat{\mathbf{K}}_k)^T\hat{\Sigma}_t^{k+1} \hat{E}_{t}^k+\sum_{t=0}^{M-1}\eta_k^2 \hat{\Sigma}_t^{k+1}\nabla_t f(\hat{\mathbf{K}}_k)\hat{\Lambda}_{t}^{k}\nabla_t f(\hat{\mathbf{K}}_k)^T\Bigg)\\
%		&\leq tr\Bigg(-\eta_k\sum_{t=0}^{M-1} \nabla_t f(\hat{\mathbf{K}}_k)\nabla_t f(\hat{\mathbf{K}}_k)^T+\eta_k\sum_{t=0}^{M-1} \frac{||\hat{\Sigma}_t^{k+1}-\hat{\Sigma}_t^k||}{\sigma_{\min}^\Sigma}\nabla_t f(\hat{\mathbf{K}}_k)\nabla_t f(\hat{\mathbf{K}}_k)^T\\
%		&+\eta_k^2\sum_{t=0}^{M-1} ||\hat{\Lambda}_{t}^k||(||\hat{\Sigma}_t^{k+1}-\hat{\Sigma}_t^k||+||\hat{\Sigma}_t^k||)\nabla_t f(\hat{\mathbf{K}}_k)\nabla_t f(\hat{\mathbf{K}}_k)^T\Bigg)\\
		\leq tr\Bigg(\sum_{t=0}^{M-1}l_{t}^k\nabla_t f(\hat{\mathbf{K}}_k)\nabla_t f(\hat{\mathbf{K}}_k)^T\Bigg)
	\end{equation}\label{f_2}
with
\begin{equation}
	l_{t}^k=-\eta_k+\eta_k\frac{||\hat{\Sigma}_t^{k+1}-\hat{\Sigma}_t^k||}{\sigma_{\min}^\Sigma}+\eta_k^2||\hat{\Lambda}_{t}^k||(||\hat{\Sigma}_t^{k+1}-\hat{\Sigma}_t^k||+||\hat{\Sigma}_t^k||).\nonumber
\end{equation}
	As $\eta_k\leq c_3^k$, we can get that $||\hat{A}_{t}^{k+1}||\leq||\hat{A}_{t}^k||+\frac{1}{2}$ and $||\hat{K}_t^k-\hat{K}_t^{k+1}||\leq 1$, and hence we have
%	 $\eta_k\leq\min\{\frac{1}{2||CA||\frac{f(\hat{\mathbf{K}}_k)}{\sigma_{\min}^{\mathbf{A^{-T}QA^{-1}}}}\sqrt{Mc_2^{-1}\big(\max_k\{f(\hat{\mathbf{K}}_k)-f({\mathbf{K}}^*)\}\big)}},\\ \frac{1}{\frac{f(\hat{\mathbf{K}}_k)}{\sigma_{\min}^{\mathbf{A^{-T}QA^{-1}}}}\sqrt{Mc_2^{-1}\big(\max_k\{f(\hat{\mathbf{K}}_k)-f({\mathbf{K}}^*)\}\big)}}\}$, which implies 
	\begin{align}
		||\hat{\Sigma}_t^{k+1}-\hat{\Sigma}_t^k||\leq\eta_k \sum_{i=0}^{M-1}\big(\frac{\rho_k^{2M}-1}{\rho_k^2-1}(2\rho_k+1)\big)||\Sigma||||CA||||\nabla_i f(\hat{\mathbf{K}}_k)||\leq\frac{1}{2}, %\notag\\
%		\leq&\eta_k\frac{f(\hat{\mathbf{K}}_k)}{\sigma_{\min}^{\mathbf{A^{-T}QA^{-1}}}}\sqrt{Mc_2^{-1}\big(\max_k\{f(\hat{\mathbf{K}}_k)-f({\mathbf{K}}^*)\}\big)}\big((\frac{\rho_k^{2M}-1}{\rho_k^2-1})(2\rho_k+1)\big)||\Sigma||||CA||\notag\\
		\label{Sigma}
	\end{align}
where the first inequality is due to Lemma \ref{lePSi} and the second one is by \eqref{ieqK} and the definition of $c_3^k$. Using \eqref{bdsigma}, \eqref{Sigma} and the inequality  $||\hat{\Lambda}_t^k||\leq\sigma_{\max}^\mathbf{CQC^T+R}+||CA||^2\mathcal{A}(\hat{\mathbf{K}}_k)$, $l_{t}^k$ can be bounded as
	\begin{align}
		l_{t}^k\leq -\frac{\eta_k}{2}+\eta_k^2(\sigma_{\max}^\mathbf{CQC^T+R}+||CA||^2\mathcal{A}(\hat{\mathbf{K}}_k))(\frac{1}{2}+\mathcal{B}(\hat{\mathbf{K}}_k)), ~~~t\in\mathbb{T}.\label{ieqlt}
	\end{align}
	Substituting $\eta_k\leq c_4^k$ into \eqref{ieqlt}, we obtain $l_t^k\leq -\frac{\eta_k}{4}$.	According to Lemma \ref{lpl}, we have
	\begin{align}
			f(\hat{\mathbf{K}}_{k+1})-f(\hat{\mathbf{K}}_k)
			\leq -\frac{\eta_k}{4c_1} \big(f(\hat{\mathbf{K}}_k)-f({\mathbf{K}}^*)\big).\label{ieq6}
	\end{align}
	 Subtracting $f({\mathbf{K}}^*)-f(\hat{\mathbf{K}}_k)$ into both sides of \eqref{ieq6}, we get
	\begin{align}
		f(\hat{\mathbf{K}}_{k+1})-f({\mathbf{K}}^*)\leq(1-\frac{\eta_k}{4c_1})\big(f(\hat{\mathbf{K}}_k)-f({\mathbf{K}}^*)\big), ~~~k\in\mathbb{N}.\label{monodes}
	\end{align}
		According to \eqref{monodes}, it holds that  $\max_k\{f(\hat{\mathbf{K}}_k)\}=f(\hat{\mathbf{K}}_0)$. Hence, using \eqref{ieqK}, $\rho_k$ is bounded by $f(\mathbf{K}_0)$, and $\max_{k\in\mathbb{N},t\in\mathbb{T}}\{||A_t^k||\}<||A||+||CA||\mathcal{C}(\mathbf{K}_0)<+\infty$. Therefore, $c_3$, $c_4$ are well defined, and we can take a small constant as $\eta$. Setting $\eta_k=\eta, k\in\mathbb{N}$ and using \eqref{monodes}, we obtain \eqref{l1}.
\end{proof}

	 {By Theorem \ref{th1}, we have $f(\mathbf{K}_k)-f(\mathbf{K}^*)\leq \epsilon$ if $k\geq \frac{\ln(\epsilon)-\ln(f(\mathbf{K}_0)-f(\mathbf{K}^*))}{\ln(1-\frac{\eta}{4c_1})}$. Using this conclusion, we can estimate the number of iterations for different accuracy requirement. Based on Theorem \ref{th1}, we may study the SGD method to solve finite-horizon KF problem with unknown noise covariance matrices.}

\section{SGD method}\label{sgdconv}

In this section, we will give the framework and analysis of SGD method for solving Problem 1 under the scenario that covariance matrices $P_0$, $\{Q_t\}_{t\in\mathbb{T}}$ $\{R_{t+1}\}_{t\in\mathbb{T}}$ are all unknown. First of all, we determine the estimation $\hat{\nabla} f(\mathbf{K})$ of $\nabla f(\mathbf{K})$, and then we analyze the error bound caused by  gradient estimation $\hat{\nabla} f(\mathbf{K})$. Finally, we analyze the convergence and sample complexity of the SGD method.

\subsection{Estimation of $\nabla f(\mathbf{K})$}

In the setting of known system parameters, an unbiased estimation of $\nabla f(\mathbf{K})$ can be formulated by model parameters $A,C$ and the observation of systems \eqref{eq0000}. For this, we first rewrite $\{\hat{y}_t^n\}_{n=1}^N$ of (\ref{dey}) as one without using state sequence $\{x_t\}$.
\begin{lemma}\label{ly}
	$\{\hat{y}_t^n\}_{n=1}^N,\ t=1,\dots,M$, defined in Problem 1, can be rewritten as
	\begin{align}
		\hat{y}_t^n&=CA^n\prod_{i=0}^{t-1}A_i\hat{x}_0+\sum_{i=0}^{t-1}CA^n\prod_{j=i+1}^{t-1}A_{j}K_iy_{i+1},~~ t\in\mathbb{T}\label{eq9}.
	\end{align}
\end{lemma}

\begin{proof}
	By the definition of $\{\hat{y}_t^n\}_{n=1}^N$, we just need to prove
	\begin{equation}
		\begin{aligned}
			&\hat{x}_t=\prod_{i=0}^{t-1}A_i\hat{x}_0+\sum_{i=0}^{t-1}\prod_{j=i+1}^{t-1}A_{j}K_iy_{i+1},
			t=1,\dots,M.
		\end{aligned}
	\end{equation} 
For $t=1$, we have
	\begin{align*}
		\hat{x}_1&=A\hat{x}_0+K_0(y_1-\hat{y}_0)\\
		&=A\hat{x}_0+K_0(y_1-CA\hat{x}_0)\\
		&=A_0\hat{x}_0+K_0y_1.
	\end{align*}
	%Using induction to complete the rest proof. 
	Suppose that $\hat{x}_t=\prod_{i=0}^{t-1}A_i\hat{x}_0+\sum_{i=0}^{t-1}\prod_{j=i+1}^{t-1}A_{j}K_i{y_{i+1}}$ holds for some $t$ with $M>t\geq1$. Then, we have
	\begin{align}
		\hat{x}_{t+1}=&A^{t+1}\hat{x}_0+\sum_{i=0}^{t}A^{t-i}K_i(y_{i+1}-\hat{y}_i^1)\notag\\
		=&A^{t+1}\hat{x}_0+\sum_{i=0}^{t-1}A^{t-i}K_t(y_{i+1}-\hat{y}_i^1)+K_{t}(y_{t+1}-\hat{y}_{t}^1)\notag\\
		=&A^{t+1}\hat{x}_0+\sum_{i=0}^{t-1}A^{t-i}K_i(y_{i+1}-\hat{y}_i^1)+K_{t}y_{t+1}-K_{t}\Bigg(CA^{t+1}x_0+\sum_{i=0}^{t-1}CA^{t-i}K_i(y_{i+1}-\hat{y}_i^1)\Bigg)\notag\\
		=&A_{t}A^{t}\hat{x}_0+\sum_{i=0}^{t-1}A_tA^{t-i-1}K_i(y_{i+1}-\hat{y}_i^1)+K_ty_{t+1}\notag\\
		=&A_t\hat{x}_t+K_ty_{t+1},\label{xeq}
	\end{align}
	 Multiplying both sides of \eqref{xeq} by $CA^n$ and substituting  $\hat{x}_t=\prod_{i=0}^{t-1}A_i\hat{x}_0+\sum_{i=0}^{t-1}\prod_{j=i+1}^{t-1}A_{j}K_i(y_{i+1})$ into \eqref{xeq}, we can  completes the proof.
\end{proof}

{Generating $L$ samples $\{y_{i+n+1}(l), i\in\mathbb{T},n=1,\cdots,N\}, l=1,\cdots,L$ by running system \eqref{eq0000}} for $L$ times, and it is clear that  $\{y_{i+n+1}(l), i\in\mathbb{T},n=1,\cdots,N\},l=1,\cdots,L$ are mutually independent. Let $e_{i+1}^n:=y_{i+n+1}-\hat{y}_{i+1}^n, i \geq t, i,t\in\mathbb{T},$ and $e_{i+1}^n(l):=y_{i+n+1}(l)-\hat{y}_{i+1}(l)^n, l=1,\dots,L, i \geq t, i,t\in\mathbb{T},$ where $\hat{y}_{i+1}^n(l)$ is generated by $\{y_{i+n+1}\}$ and Kalman gains $\mathbf{K}$ in the same way as $\hat{y}_{i+1}^n$. According to the above notations and Lemma \ref{ly}, we have the following corollary.

\begin{corollary}\label{co_denote_e}
	The gradient of $\sum_{l=1}^{L}e_{i+1}^{nT}(l)e_{i+1}^n(l)$ with respect to $K_t, t\geq i, i,t\in\mathbb{T},$ can be represented as %\textcolor{red}{
%\begin{equation}
		\begin{align}
			&\nabla_t \sum_{l=1}^{L}e_{i+1}^{nT}(l)e_{i+1}^n(l)\notag\\
			&={-}2\Bigg(\sum_{l=1}^{L}\bigg((CA^n\prod_{j=t+1}^{i}A_j)^Te_{i+1}^{n}(l)y_{t+1}(l)^T
			-\sum_{j=0}^{t-1}(CA^{n}\prod_{k=t+1}^{i}A_k)^Te_{i+1}^n(l)\notag\\
			&\hphantom{=}\cdot\bigg(\prod_{k=j+1}^{t-1}A_kK_jy_{j+1}(l))^T(CA)^T
			-(CA^n\prod_{j=t+1}^{i}A_j)^Te_{i+1}^n(l)(\prod_{j=0}^{t-1}A_j\hat{x}_0)^T(CA)^T\bigg)\Bigg).\label{eq11}
		\end{align}
%	\end{equation}
%}
\end{corollary}
\begin{proof}
Note that 
	\begin{align*}
		\hat{y}_{i+1}^n(l)&=\sum_{j=0}^{3}U_j
	\end{align*}
%\textcolor{red}
{with 
	\begin{align*}
		U_0&=\sum_{j=t+1}^{i}CA^n\prod_{k=j+1}^{i}A_kK_jy_{j+1}(l),\\
		U_1&=CA^n\prod_{j=t+1}^{i}A_jK_ty_{t+1}(l),\\
		U_2&=\sum_{j=0}^{t-1}CA^n\prod_{k=j+1}^{i}A_kK_jy_{j+1}(l),\\
		U_3&=CA^n\prod_{j=0}^{i}A_j\hat{x}_0.
	\end{align*}
}
As $U_0$ has nothing to do with $K_t$, we have \eqref{eq11} by taking the gradient of the other three terms.
\end{proof}

Using Corollary \ref{co_denote_e}, we propose the estimation of $\nabla_t f(\mathbf{K})$ 
\begin{align}
	\hat{\nabla}_t f(\mathbf{K})
	=\frac{1}{L}\sum_{i=t}^{M-1}\sum_{l=1}^{L}\sum_{n=1}^{N}\nabla_t e_{i+1}^{nT}(l)e_{i+1}^n(l),~~~ t\in\mathbb{T},\label{eg}
\end{align}
where $\nabla_t e_{i+1}^{nT}(l)e_{i+1}^n(l)$ denotes the gradient of $ e_{i+1}^{nT}(l)e_{i+1}^n(l)$ with respect to $K_t$.

\subsection{Local Lipschitz continuity of $f(\mathbf{K})$}

In zeroth-order PG method, the local Lipschitz continuity of optimization function is used to analyze the error caused by gradient estimation, and the analysis method is similar in several variants of LQR problem \cite{pmlr-v80-fazel18a,doi:10.1137/20M1382386,9130755,10156641,Talebi2023DatadrivenOF}. Referencing Lemma 4.7 of \cite{doi:10.1080/14697688.2020.1813475}, we propose the following theorem to characterize the local Lipschitz continuity of $f(\mathbf{K})$.
 
\begin{lemma}\label{lllf}
	Under assumptions \ref{assumtion-1}-\ref{as3} and for any $\mathbf{K},\mathbf{K}' \in\mathbb{R}^{m\times MN}$, it holds % $f(\mathbf{K})$ satisfies
	\begin{align*}
		| f(\mathbf{K})- f(\mathbf{K}')|\leq l_{cost\mathbf{K}}||\mathbf{K}-\mathbf{K}'||,~~ t\in\mathbb{T}
	\end{align*}
	under condition  $||K_{t}-K_{t}'||\leq\min\{\frac{1}{2||CA||},1\}, t\in\mathbb{T}$, and $l_{cost\mathbf{K}}$ is bounded by a polynomial of $f(\mathbf{K})$.
\end{lemma}

\begin{proof}
	By Lemma \ref{lsmooth}, we have
	\begin{align*}
		|f(\mathbf{K})-f(\mathbf{K}')|
%		=&\Bigg|tr\Bigg(2\sum_{t=0}^{M-1}\delta K_{t}^T\Sigma_{t} E_t'+\sum_{t=0}^{M-1} \Sigma_{t}\delta K_{t}\Lambda_{t}'\delta K_{t}^T\Bigg)\Bigg|\\
		\leq&2\sum_{t=0}^{M-1}||K_{t}-K_{t}'||(||E_t-E_t'||+||E_t||)tr(\Sigma_t)\\
		&+\sum_{t=0}^{M-1}tr(\Sigma_t)(||\Lambda_t'-\Lambda_t||+||\Lambda_t||)||K_t-K_t'||. 
	\end{align*}
%where the last inequality is due to \eqref{bdsigma}.
	In order to determine the local Lipschitz constant of $f(\mathbf{K})$, the ones of $P_{t}$ and $\Sigma_t,\ t\in\mathbb{T}$ need to be specified first. According to \eqref{eqly}, we have
	\begin{align}
		||P_{t}-P_{t}'||
%		=&||A_{t-1}P_{t-1}A_{t-1}^T+(I-K_{t-1}C)Q_{t-1}(I-K_{t-1}C)^T+K_{t-1}R_iK_{t-1}^T\notag\\
%		&-A_{i-1}'P_{t-1}'A_{t-1}'^T+(I-K_{t-1}'C)Q_{t-1}(I-K_{t-1}'C)^T+K_{t-1}'R_tK_{t-1}'^T||\notag\\
		\leq&||A_{t-1}P_{t-1}A_{t-1}^T-A_{t-1}'P_{t-1}'A_{t-1}'^T||+||K_{t-1}(R_t+CQ_{t-1}C^T)K_{t-1}^T\notag\\
		&-K_{t-1}'(R_t+CQ_{t-1}C^T)K_{t-1}'^T||
		+2||K_{t-1}CQ_{t-1}-K_{t-1}'CQ_{t-1}||,\label{ieqP}
	\end{align}
	for $t=1,\dots,M-1$. Using $||K_{t}-K_{t}'||\leq\min\{\frac{1}{2||CA||},1\},t\in\mathbb{T},$ and $\rho=\max_{t\in\mathbb{T}}\{\max\{||A_t||+\frac{1}{2},1+b\}\}, b>0$, we have
	\begin{align*}
		||P_{t}-P_{t}'||
		\leq&\bigg((2||K_{t-1}||+1)||R_t+CQ_{t-1}C^T||+2||CQ_{t-1}||+(2\rho+1)\mathcal{A}(\mathbf{K})||CA||\bigg)||K_{t-1}-K_{t-1}'||\\
		&+\rho^2||P_{t-1}-P_{t-1}'||,
	\end{align*}
	where the inequality holds by \eqref{ieqp}. When $t=0$, we have $||P_t-P_t'||=||P_0-P_0||=0$. According to the above analysis, we can get
	%\begin{equation}
		\begin{align}
			||P_{t}-P_{t}'||\leq&\sum_{i=1}^{t}c_{ti}||K_{i-1}-K_{i-1}'||,\ ~~ \forall t\in\mathbb{T}\label{ieq1}
		\end{align}
	%\end{equation}
	with $c_{ti}=\rho^{2t-2i}\big((2||K_{i-1}||+1)||R_i+CQ_{i-1}C^T||+2||CQ_{i-1}||+(2\rho+1)\mathcal{A}(\mathbf{K})||CA||\big)$. 
	Substituting \eqref{ieq1} into $||E_t-E_t'||$ and $||\Lambda_t-\Lambda_t'||$, respectively, we have
	\begin{align}
		||E_t-E_t'||
%		\leq&||(K_t-K_t')(R_{t+1}+CQ_tC^T)-A(P_{t}-P_t')(CA)^T\notag\\
%		&+K_tCAP_t(CA)^T-K_t'(CAP_t'(CA)^T)||\notag\\
%		\leq&(\sigma_{\max}^\mathbf{R+CQC^T}+||CA||^2\frac{f(\mathbf{K})}{\sigma_{\min}^\Sigma})||K_t-K_t'||+\big(||A||||CA||\notag\\
%		&+||CA||^2(1+||K_t||)\big)\sum_{i=1}^{t}c_{ti}||K_{i-1}-K_{i-1}'||\notag\\
		\leq&(\sigma_{\max}^\mathbf{R+CQC^T}+||CA||^2\mathcal{A}(\mathbf{K}))||K_t-K_t'||+\big(||A||||CA||\notag\\
		&+||CA||^2(1+||K_t||)\big)c_{\mathbf{K}}|||\mathbf{K}-\mathbf{K}'|||,\label{ieq2}
	\end{align}
and
\begin{align}
	||\Lambda_t-\Lambda_t'||\leq&||CA||^2||P_t-P_t'||
	\leq||CA||^2c_{\mathbf{K}}|||\mathbf{K}-\mathbf{K}'|||, \label{ieqLam}
\end{align}
	where $c_{\mathbf{K}}=\rho^{2M}\big((2\max_{t\in \mathbb{T}}||K_t||+1)\sigma_{\max}^\mathbf{R+CQC^T}+2||C||\sigma_{\max}^\mathbf{Q}+(2\rho+1)\mathcal{A}(\mathbf{K})||CA||\big)$, and $\sigma_{\max}^\mathbf{Q}:=\max_{t\in \mathbb{T}}\{\sigma_{\max}^{Q_t}\}$.
	
	Letting $c_5=M||CA||^2c_{\mathbf{K}}$,  $c_6=(\sigma_{\max}^\mathbf{R+CQC^T}+||CA||^2\mathcal{A}(\mathbf{K}))+M\big(||A||||CA||+||CA||^2(1+\max_{t\in \mathbb{T}}||K_t||)\big)c_{\mathbf{K}}$ and $c_7=c_6+\sqrt{Mc_2^{-1}(f(\mathbf{K})-f(\mathbf{K}^*))}$, we have
	\begin{align*}
		|f(\mathbf{K})-f(\mathbf{K}')|
%		,i,t\in\mathbb{T}, i\geq t
%		\leq&2\sum_{t=0}^{M-1}\frac{f(\mathbf{K})}{\sigma_{\min}^{\mathbf{A^{-T}QA^{-1}}}}||K_{t}-K_{t}'||(||E_t-E_t'||+||E_t||)\\
%		&+\sum_{t=0}^{M-1}\frac{f(\mathbf{K})}{\sigma_{\min}^{\mathbf{A^{-T}QA^{-1}}}}(||\Lambda_t'-\Lambda_t||+||\Lambda_t||)||K_t-K_t'||\\
%		\leq& 2\sum_{t=0}^{M-1}c_7\frac{f(\mathbf{K})}{\sigma_{\min}^{\mathbf{A^{-T}QA^{-1}}}}||K_{t}-K_{t}'||\\
%		&+\sum_{t=0}^{M-1}\frac{f(\mathbf{K})}{\sigma_{\min}^{\mathbf{A^{-T}QA^{-1}}}}(||\Lambda_{\max}||+c_5)||K_t-K_t'||\\
		\leq&l_{cost\mathbf{K}}|||\mathbf{K}-\mathbf{K}'|||
	\end{align*}
	with $l_{cost\mathbf{K}}= \mathcal{B}(\mathbf{K})(2 c_7+\sigma_{\max}^\mathbf{CQC^T+R}+||CA||^2\mathcal{A}(\mathbf{K})+c_5).$ In the above, the inequality is due to \eqref{bdsigma}, \eqref{ieq2} and \eqref{ieqLam}.
\end{proof}
\begin{remark}\label{remark-1}
	According to \eqref{ieqK}, we can replace $\rho$ by  $\max\{||A||+||CA||\mathcal{C}(\mathbf{K})+\frac{1}{2},1+b\}$ in $c_{\mathbf{K}}$, $c_6$ and $l_{cost\mathbf{K}}$ such that the local Lipschitz constant of $f(\mathbf{K})$ can be bounded by the polynomial of $f(\mathbf{K})$.
\end{remark}

\subsection{Convergence of SGD method}
According to the gradient estimation $\hat{\nabla} f(\mathbf{K})$, we propose the SGD method to solve Problem 1.
\begin{breakablealgorithm}%[H]
	\caption{SGD method for Kalman filter}%算法标题
	\begin{algorithmic}[1]%一行一个标行号
		\State \textbf{Input}: $\mathbf{K}_0$, $M$, $C$, $A$, $\eta$, number of iterations $V$, sample size $L$.
		\State \textbf{Initialization}: Generate $L$ samples {$Y(l)=\{y_{t+n+1}(l),t\in\mathbb{T},n=1,\cdots,N\}, l=1,2,\cdots,L$.}
		\For{$k=0,1,...,V-1$}
		\For{$l=1,2,\cdots,L$}
        \State
		Calculate $\{\hat{y}_{t+1}^n(l),t\in\mathbb{T},n=1,\cdots,N\},$ by $\mathbf{K}_k$, $Y$ and \eqref{eq9}.
		\EndFor
		\State 
		Calculate $\hat{\nabla}f(\mathbf{K}_{k})$ by \eqref{eg}
		\State
		$\mathbf{K}_{k+1}\leftarrow \mathbf{K}_{k}-\eta \hat{\nabla} f(\mathbf{K}_{k})$
		\EndFor
	\end{algorithmic}
\end{breakablealgorithm}

We now analyze the convergence of Algorithm 1. 
\begin{lemma}\cite{doi:10.1080/14697688.2020.1813475}\label{le1}
	Let $X$ be a subgaussian random vector in $\mathbb{R}^n$. If there exists $K\geq 1$ such that 
	\begin{align*}
		||\langle X,x\rangle||_{\psi_2}\leq K||\langle X,x \rangle||_{L_2}, \forall x \in \mathbb{R}^n,
	\end{align*}
	then for every $u,L>0$, the inequality
	\begin{align*}
		||D_L-D||\leq CK^2\Big(\sqrt{\frac{n+u}{L}}+\frac{n+u}{L}\Big)||D||
	\end{align*}
	holds with probability at least $1-2e^{-2u}$, where $C$ is a positive constant, $D=\mathbb{E}[XX^T]$, and $D_L$ is the sample version of $D$ computed
	from the $i.i.d$ samples $X_1,\dots,X_L$ as follows:
	\begin{equation}
		D_L=\frac{1}{L}\sum_{l=1}^{L}X_lX_l^T.\notag
	\end{equation}
\end{lemma}

Using the above lemma of \cite{doi:10.1080/14697688.2020.1813475}, we can characterize the error bound of $\hat{\nabla}_t f(\mathbf{K}),t\in\mathbb{T}$.

\begin{lemma}\label{lnablaest}
	The following result holds
\begin{equation}
	\nonumber
	\mathbb{P}\bigg\{||\hat{\nabla}_t f(\mathbf{K})-\nabla_t f(\mathbf{K})||\leq \frac{8}{3}C'(\sqrt{\frac{z+u}{L}}+\frac{z+u}{L}), \forall\ \mathbf{K}\in \mathbb{R}^{m\times MN}, \forall\ t\in\mathbb{T}\bigg\}\geq 1-2e^{-2u},\\
\end{equation}
where $C'$ is a positive constant bounded by a polynomial of $f(\mathbf{K})$ and {$z=N(N+M+1)+m(N+M)$}.  
\end{lemma}
\begin{proof}
	
		According to Lemma \ref{ly} and \eqref{KFy}, we have
	\begin{align}
		y_{i+n+1}=&\sum_{k=0}^{i+n}CA^k\omega_{i+n-k}+v_{i+n+1}+CA^{i+n+1}x_0,\label{eqy}\\
		\hat{y}_{i+1}^n=&CA^n\prod_{j=0}^{i}A_j\hat{x}_0+\sum_{j=0}^{i}CA^n\prod_{k=j+1}^{i}A_{k}K_jy_{j+1}\notag\\
		=&CA^n\prod_{j=0}^{i}A_j\hat{x}_0+\sum_{j=0}^{i}CA^n\prod_{k=j+1}^{i}A_{k}K_j\Bigg(\sum_{l=0}^{j}CA^l\omega_{j-l}+v_{j+1}+CA^{j+1}x_0\Bigg)\notag\\
		=&CA^n\prod_{j=0}^{i}A_j\hat{x}_0+\sum_{j=0}^{i}CA^n\prod_{k=j+1}^{i}A_{k}K_jCA^{j+1}x_0\notag\\
		&+\sum_{j=0}^{i}\sum_{k=j}^{i}CA^n\prod_{l=k+1}^{i}A_lK_kCA^{k-j}\omega_j+\sum_{j=1}^{i+1}CA^n\prod_{k=j}^{i}A_kK_{j-1}v_j\label{eqyj}
	\end{align}
	with $i\in\mathbb{T}, n=1,\dots,N$. Substituting \eqref{eqy} and \eqref{eqyj} into $y_{i+n+1}-\hat{y}_{i+1}^n$, we have
	\begin{equation}
		\begin{aligned}
			y_{i+n+1}-\hat{y}_{i+1}^n=&CA^n\prod_{j=0}^{i}A_je_0+\sum_{j=0}^{i} CA^n\prod_{k=j+1}^{i}A_k(I_N-K_jC)\omega_j\\
			&+\sum_{j=i+1}^{i+n}CA^{i+n-j}\omega_j-\sum_{j=1}^{i+1}CA^n\prod_{k=j}^{i}A_kK_{j-1}v_j+v_{i+n+1},\label{err_y}
		\end{aligned}
	\end{equation}
	where $e_0$ is $x_0-\hat{x}_0$.
	Letting  $X^T=(e_0^T,\omega_0^T,\dots,\omega_{M+N-1}^T,v_1^T,\dots,v_{M+N}^T)$ and taking \eqref{err_y} into the optimization function of Problem 1, we get 
	\begin{equation}
		{f_1(\mathbf{K})}:=\mathbb{E}\Bigg[\sum_{i=0}^{M-1}\sum_{n=1}^{N}(y_{i+n+1}-\hat{y}_{i+1}^n)^T(y_{i+n+1}-\hat{y}_{i+1}^n)\Bigg]=\mathbb{E}\Bigg[X^T\sum_{i=0}^{M-1}\ \sum_{n=1}^{N}\Lambda_i^n\Lambda_i^{nT}X\Bigg],\label{eqf2}
	\end{equation}
	where
	\begin{equation}
		\nonumber
		\begin{aligned}
			\Lambda_i^{nT}=&\bigg(CA^n\prod_{j=0}^{i}A_j, CA^n\prod_{j=1}^{i}A_j(I_N-K_0C),\dots,CA^n(I_N-K_iC),CA^{n-1},\dots,C,\\
			&0_{m\times N(M+N-i-n-1)},CA^n\prod_{j=1}^{i}A_jK_0,\dots,K_i,0_{m\times m(n-1)},I_m,0_{m\times m(M+N-i-n-1)}\bigg)
		\end{aligned}
	\end{equation}
	is a $\mathbb{R}^{m\times (N(N+M)+m(N+M)+1)}$-valued matrix.

Define $\delta_t \Lambda_i^{nT},i\geq t,i,t\in\mathbb{T}$, as
	\begin{equation}
		\begin{aligned}
			\delta_t \Lambda_i^{nT}:=&\big(-CA^n\prod_{j=t+1}^{i}A_j\delta K_t CA\prod_{j=0}^{t-1}A_j,-CA^n\prod_{j=t+1}^{i}A_j\delta K_tCA \prod_{j=1}^{t-1}A_j(I_N-K_0C),\\
			&\dots,-CA^n\prod_{j=t+1}^{i}A_j\delta K_tC,0_{m\times N(M+N-t-1)},-CA^n\prod_{j=t+1}^{i}A_j\delta K_t CA\prod_{j=1}^{t-1}A_jK_0,\\
			&\dots,CA^n\prod_{j=t+1}^{i}A_j\delta K_t,
			0_{m\times m(M+N-t-1)}\big), t\neq0\label{eq10},\\
			\delta_t \Lambda_i^{nT}:=&\big(-CA^n\prod_{j=1}^{i}A_j\delta K_0 CA,-CA^n\prod_{j=1}^{i}A_j\delta K_0C,0_{m\times N(M+N-1)},\\
			&CA^n\prod_{j=1}^{i}A_j\delta K_0,
			0_{m\times m(M+N-1)}\big), t=0,
		\end{aligned}
	\end{equation}
	where $\delta K_t$ is any $\mathbb{R}^{N\times m}$-valued matrix satisfying 
		\begin{equation}
		\nonumber
		||\delta K_t||_{F}=1.
	\end{equation}
	It is easy to see
	\begin{equation}
		tr(\nabla_t f(\mathbf{K})\delta K_t)=2\mathbb{E}\Bigg[X^T\sum_{i=t}^{M-1}\ \sum_{n=1}^{N}\Lambda_i^n\delta_t\Lambda_i^{nT}X\Bigg].\label{diff2}
	\end{equation}
Let $X_l^T=(e_0(l)^T, \omega_0(l)^T, \dots,\omega_{M+N-1}(l)^T,v_1(l)^T,\dots,v_{M+N}(l)^T), ~l=1,\dots,L,$ be the $i.i.d$ samples of $X$. According to \eqref{eg},\eqref{eqf2} and \eqref{diff2}, we have
	\begin{equation}
		tr(\hat{\nabla}_t f(\mathbf{K})\delta K_t)=2{\frac{1}{L}}\sum_{l=1}^{L}\bigg(X_l^T\sum_{i=t}^{M-1}\ \sum_{n=1}^{N}\Lambda_i^n\delta_t\Lambda_i^{nT}X_l\bigg).\label{diff}
	\end{equation}
Defining the covariance matrix of $X$ as $D_X$ and by Lemma \ref{le1}, \eqref{diff} and the properties of gradient, the following inequality 
	\begin{align}
		tr\Bigg((\hat{\nabla}_t f(\mathbf{K})-\nabla_t f(\mathbf{K}))\delta K_t^T\Bigg)
		=&2tr\Bigg(\frac{1}{L}\sum_{l=1}^{L}\sum_{n=1}^{N}\sum_{i=t}^{M-1}\delta_t\Lambda_i^{nT}X_lX_l^T \Lambda_i^n-\sum_{n=1}^{N}\sum_{i=t}^{M-1}\delta_t\Lambda_i^{nT}D_{X} \Lambda_i^n\Bigg)\notag\\
		\leq&
		2z||\frac{1}{L}\sum_{l=1}^{L}X_lX_l^T-D_{X}||\sum_{n=1}^{N}\sum_{i=t}^{M-1}||\delta_t\Lambda_i^{nT}|| ||\Lambda_i^n||\notag\\
		\leq& 2zK^2C(\sqrt{\frac{z+u}{L}}+\frac{z+u}{L})||D_X||\sum_{n=1}^{N}\sum_{i=t}^{M-1}||\delta_t\Lambda_i^{nT}|| ||\Lambda_i^n||\label{ieq3}
	\end{align}
	holds, as 
	\begin{equation}
		||\frac{1}{L}\sum_{l=1}^{L}X_lX_l^T-D_{X}||\leq K^2C\bigg(\sqrt{\frac{z+u}{L}}+\frac{z+u}{L}\bigg){||D_X||}. \label{eq_sto}
	\end{equation}
	 Dividing $\delta_t\Lambda_i^{nT}$ into two parts, we have
	\begin{equation}\label{deltaL}
		\delta_t\Lambda_i^{nT}=\Delta_t^1\Lambda_i^{nT}\Delta_t^2\Lambda_i^{nT},
	\end{equation}
	where %\textcolor{red}{
	\begin{equation}
		\begin{aligned}
			\label{DeltaL}
			\Delta_t^1\Lambda_i^{nT}=&\Bigg(\mathbf{1}_{1\times (t+2)}\otimes -CA^n\prod_{j=t+1}^{i}A_j\delta K_t, 0_{m\times m(M+N-t-1)},\\
			&\mathbf{1}_{1\times (t+1)}\otimes-CA^n\prod_{j=t+1}^{i}A_j\delta K_t ,
			0_{m\times m(M+N-t-1)}\Bigg),\\
			\Delta_t^2\Lambda_i^{nT}=&\mbox{diag}\Bigg\{  CA\prod_{j=0}^{t-1}A_j,CA \prod_{j=1}^{t-1}A_j(I_N-K_0C),
			\dots,C,
			0_{m(M+N-t-1)\times N(M+N-t-1)},\\ &CA\prod_{j=1}^{t-1}A_jK_0,\dots,I_m,0_{m(M+N-t-1)\times m(M+N-t-1)}\Bigg\}, ~~t\neq 0,\\
			\Delta_t^2\Lambda_i^{nT}=&\mbox{diag}\big\{  CA,C,0_{m\times N(M+N-1)}, I_{m},0_{m{(M+N-1)}\times m(M+N-1)}\big\}, t=0.
		\end{aligned}
	\end{equation}%}
	{According to} \eqref{bdsigma}, $||\Delta_t^1\Lambda_i^{nT}||\leq(2t+3)||CA^n||\sqrt{ \frac{\mathcal{B}(\mathbf{K})}{\sigma_{\min}^{\Sigma}}}{||\delta K_t||_F}$. Using \eqref{eqf2}, we have
	\begin{equation}
		||\Delta_t^2\Lambda_i^{nT}||\leq \sqrt{\frac{{f_1}(\mathbf{K})}{\sigma_{\min}^{D_X}}},
	\end{equation}
	 if $\sigma_{\min}^{P_0}>0$. If $\sigma_{\min}^{P_0}=0$, we have 
 \begin{equation}
 	||\Delta_t^2\Lambda_i^{nT}||\leq(1+||A||)\sqrt{ \frac{{f_1}(\mathbf{K})}{\sigma_{\min}^{\mathbf{Q,\mathbf{R}}}}},
 \end{equation}
 by using inequality $||CA\prod_{j=0}^{t-1}A_j||< ||A||||CA \prod_{j=1}^{t-1}A_j(I_N-K_0C)||$, where $\sigma_{\min}^{\mathbf{Q,\mathbf{R}}}:=\min\{\sigma_{\min}^{Q_t},\sigma_{\min}^{R_{t+1}},$ $t\in\mathbb{T}\}$. 
 Let 
 \begin{align*}
&c_8(\mathbf{K})=2({M+2})||CA^n||\sqrt{\frac{\mathcal{B}(\mathbf{K})}{\sigma_{\min}^{\Sigma}}},\\
&c_9(\mathbf{K})={\min}\Bigg\{(1+||A||)\sqrt{ \frac{{f_1}(\mathbf{K})}{\sigma_{\min}^{\mathbf{Q,\mathbf{R}}}}},\sqrt{\frac{{f_1}(\mathbf{K})}{\sigma_{\min}^{D_X}}}\Bigg\},
 \end{align*}
 which are the uniformly upper bounds of $\{||\Delta_t^1\Lambda_i^{nT}||,i,t\in\mathbb{T}, i\geq t\}$ and $\{||\Delta_t^2\Lambda_i^{nT}||,i,t\in\mathbb{T}, i\geq t\}$, respectively.
	 Similarly to the analysis of $||\Delta_t^2\Lambda_i^{nT}||$, we can also get $||\Lambda_i^n||\leq c_9(\mathbf{K}), i\in\mathbb{T}$. Substituting \eqref{deltaL}, $c_8(\mathbf{K}),c_9(\mathbf{K})$ into \eqref{ieq3} and by the definition of operator norm, we have
	\begin{align*}
		||\hat{\nabla}_t f(\mathbf{K})-\nabla_t f(\mathbf{K})||_F\leq K^2C'\bigg(\sqrt{\frac{z+u}{L}}+\frac{z+u}{L}\bigg), ~~  t\in\mathbb{T}
	\end{align*}
with $C'=2zCMNc_8(\mathbf{K})c_9^2(\mathbf{K})||D_X||$. Noting that \eqref{eq_sto} holds with probability at least  $1-e^{-2u}$, we have
\begin{equation}
	\mathbb{P}\{||\hat{\nabla}_t f(\mathbf{K})-\nabla_t f(\mathbf{K})||\leq K^2C'(\sqrt{\frac{z+u}{L}}+\frac{z+u}{L}), \forall \mathbf{K}\in \mathbb{R}^{m\times MN}, \forall t\in\mathbb{T}\}\geq 1-2e^{-2u}.
\end{equation}
According to \eqref{eqn0} and the definition of $\psi_2$ norm \cite{doi:10.1080/14697688.2020.1813475}, we have
	\begin{align*}
		&||\langle X,x\rangle||_{\psi_2}=\frac{2\sqrt{2}}{\sqrt{3}}\sigma,\\
		&\sigma=\sqrt{\langle x,D_X x\rangle},
	\end{align*} 
	as $X\sim\mathcal{N}(0,D_X)$. Hence, $K=\frac{2\sqrt{2}}{\sqrt{3}}$.
\end{proof}

Using \eqref{eqf2}, we can easily show that $\hat{\nabla} f(\mathbf{K}):=(\hat{\nabla}_0 f(\mathbf{K}),\cdots, \hat{\nabla}_{M-1} f(\mathbf{K}))$ is an unbiased estimation of $\nabla f(\mathbf{K})$, namely, 
	\begin{align*}
 	  \nabla f(\mathbf{K})=& \nabla \mathbb{E}[X^T\sum_{t=0}^{M-1}\  \sum_{n=1}^{N}\Lambda_t^n\Lambda_t^{nT}X]\\
 	  =&\nabla tr(D_X\sum_{t=0}^{M-1}\sum_{n=1}^{N}\Lambda_t^n\Lambda_t^{nT})\\
 	  =&\mathbb{E}[\nabla tr(XX^T\sum_{t=0}^{M-1}\sum_{n=1}^{N}\Lambda_t^n\Lambda_t^{nT})]\\
 	  =&\mathbb{E}[\nabla tr(\frac{1}{L}\sum_{l=1}^{L}X_lX_l^T\sum_{t=0}^{M-1}\sum_{n=1}^{N}\Lambda_t^n\Lambda_t^{nT})]\\
 	  =&\mathbb{E}[\hat{\nabla} f(\mathbf{K})].
	\end{align*}
The third equality is by the dominated convergence theorem.

Combined with Theorem \ref{th1}, we could analyze the convergence rate and sample complexity of Algorithm 1.

\begin{theorem}\label{thsgd}
	Let assumptions \ref{assumtion-1}-\ref{as3} hold and  $\{\mathbf{K}_k\}$ be the  gain sequence generated by Algorithm 1 with step size $\eta$ designed by \eqref{eta0}. With enough samples  $L=O\big(\frac{1}{\epsilon^2}ln(\frac{1}{\delta})\big)$ and as long as $f(\mathbf{K}_k)-f(\mathbf{K}^*)\geq \epsilon$, Algorithm 1 is linearly convergent in the sense  
	\begin{align}
		f(\mathbf{K}_k)-f(\mathbf{K}^*)\leq (1-\alpha)^k \big(f(\mathbf{K}_0)-f(\mathbf{K}^*)\big) \label{eq_end}
	\end{align}
	with high probability $1-\delta$.
\end{theorem}
\begin{proof}
	Suppose that $||\hat{\nabla}_t f(\mathbf{K})-\nabla_t f(\mathbf{K})||\leq \frac{8}{3}C'(\sqrt{\frac{z+u}{L}}+\frac{z+u}{L})$, $\forall$ $\mathbf{K}\in \mathbb{R}^{m\times MN}$, $\forall t\in\mathbb{T},$ and $f(\mathbf{K}_0)-f(\mathbf{K}^*)\geq \epsilon$ hold. According to Theorem \ref{th1}, we have
	\begin{align*}
	f(\mathbf{K}_{1})-f(\mathbf{K}^*)\leq& (1-2\alpha)\big(f(\mathbf{K}_0)-f(\mathbf{K}^*)\big)-f(\mathbf{K}_0+\eta\nabla f(\mathbf{K}))+f(\mathbf{K}_0+\eta\hat{\nabla} f(\mathbf{K})).
    \end{align*}
	 By Lemma \ref{lllf}, Lemma \ref{lnablaest} and { for $L\geq (\frac{16\eta C'M\sqrt{z+u}}{3\min\{1,\frac{1}{2||CA||}\}})^2$}, we can obtain
	\begin{align*}
		f(\mathbf{K}_{1})-f(\mathbf{K}^*)
		\leq&(1-2\alpha)\big(f(\mathbf{K}_0)-f(\mathbf{K}^*)\big)+\eta l_{cost\mathbf{K}_0}\sum_{t=0}^{M-1}||\nabla_t f(\mathbf{K}_0)-\hat{\nabla}_tf(\mathbf{K}_0)||\\
		\leq&(1-2\alpha)\big(f(\mathbf{K}_0)-f(\mathbf{K}^*)\big)+\eta l_{cost\mathbf{K}_0}\frac{8}{3}C'M\bigg(\sqrt{\frac{z+u}{L}}+\frac{z+u}{L}\bigg).
	\end{align*}
	Letting $L{\geq} (\frac{16\eta {\max\{l_{cost\mathbf{K}_0},1\}}C'M\sqrt{z+u}}{3\epsilon\alpha\min\{1,\frac{1}{2||CA||}\}})^2$, we have
	\begin{align}
		f(\mathbf{K}_{1})-f(\mathbf{K}^*)\leq& (1-2\alpha)\big(f(\mathbf{K}_0)-f(\mathbf{K}^*)\big)+\alpha \epsilon\notag\\
		\leq& (1-\alpha)\big(f(\mathbf{K}_0)-f(\mathbf{K}^*)\big).\label{sgdc}
	\end{align}
	If $f(\mathbf{K}_1)-f(\mathbf{K}^*)\geq\epsilon$, using the definition of $l_{cost\mathbf{K}}$ and Remark \ref{remark-1}, we have $l_{cost\mathbf{K}_0}\geq l_{cost\mathbf{K}_1}$, and hence \eqref{eq_end} holds for $k=2$. In a similar fashion, we can prove that $\{f(\mathbf{K}_k)\}$ monotonically decreases as $f(\mathbf{K_k})-f(\mathbf{K}^*)\geq \epsilon$. According to Lemma \ref{lnablaest} and letting  $L{\geq} (z+\frac{1}{2}ln\big(\frac{2}{\delta})\big)(\frac{16\eta {\max\{l_{cost\mathbf{K}_0},1\}}C'}{3\epsilon\alpha\min\{1,\frac{1}{2||CA||}\}})^2 $ and $u=\frac{1}{2}ln\big(\frac{2}{\delta})$, \eqref{eq_end} will hold with probability at least $1-\delta$ until $f(\mathbf{K}_k)-f(\mathbf{K}^*)<\epsilon$. So we have $L= O\big(\frac{1}{\epsilon^2}ln(\frac{1}{\delta})\big)$.
	This completes the proof.
\end{proof}

\begin{remark}
Noting that Lemma \ref{le1} holds for any subgaussian vector $X$,  {Lemma \ref{lnablaest} and Theorem \ref{thsgd} can be extended to the case where noise sequences $\{v_t\}, \{\omega_t\}$ proposed in \eqref{eq0000} are subgaussian. }
\end{remark}

\begin{corollary}
	Let assumptions \ref{assumtion-1}-\ref{as3} hold and  $\{\mathbf{K}_k\}$ be the gain sequence generated by Algorithm 1 with step size $\eta$ designed by \eqref{eta0}. With enough samples $L=O\big(\frac{1}{\epsilon^2}ln(\frac{1}{\delta})\big)$ and as long as $f(\mathbf{K}_k)-f(\mathbf{K}^*)\geq \epsilon$,  $||\mathbf{K}_k-\mathbf{K}^*||_F^2$ decays linearly {with probability at least $1-\delta$} 
	\begin{equation}
		||\mathbf{K}_k-\mathbf{K}^*||_F^2\leq c_{10}(1-\alpha)^k||\mathbf{K}_0-\mathbf{K}^*||_F^2, k\in\mathbb{N},
	\end{equation}
where $c_{10}=\frac{\mathcal{B}(\mathbf{K}_0)(\sigma_{\max}^\mathbf{CQC^T+R}+||CA||^2\mathcal{A}(\mathbf{K}^*))}{\sigma_{\min}^\Sigma\sigma_{\min}^\mathbf{R}}$.
\end{corollary}
\begin{proof}
	By Theorem \ref{thsgd}, we have that
	\begin{align*}
		f(\mathbf{K}_k)-f(\mathbf{K}^*)\leq (1-\alpha)^k \big(f(\mathbf{K}_0)-f(\mathbf{K}^*)\big) 
	\end{align*}
holds with high probability $1-\delta$ as long as  $f(\mathbf{K}_k)-f(\mathbf{K}^*)\geq \epsilon$. Using Lemma \ref{lsmooth}, we can obtain
\begin{align}
	f(\mathbf{K}_0)-f(\mathbf{K}^*)\leq& \mathcal{B}(\mathbf{K}_0)(\sigma_{\max}^\mathbf{CQC^T+R}+||CA||^2\mathcal{A}(\mathbf{K}^*))||\mathbf{K}_0-\mathbf{K}^*||_F^2,\label{fieK}\\
	f(\mathbf{K}_{k})-f(\mathbf{K}^*)\geq& \sigma_{\min}^\Sigma\sigma_{\min}^\mathbf{R}||\mathbf{K}_{k}-\mathbf{K}^*||_F^2,\label{Kief}
\end{align}
where the first inequality is by $||\Sigma_t^0||||\Lambda_t^*||\leq\mathcal{B}(\mathbf{K}_0)(\sigma_{\max}^\mathbf{CQC^T+R}+||CA||^2\mathcal{A}(\mathbf{K}^*)), t\in\mathbb{T}$. Using \eqref{fieK} and \eqref{Kief}, we have
\begin{equation}
	\nonumber
		||\mathbf{K}_k-\mathbf{K}^*||_F^2\leq c_{10}(1-\alpha)^k||\mathbf{K}_0-\mathbf{K}^*||_F^2, k\in\mathbb{N},
\end{equation}
where $c_{10}=\frac{\mathcal{B}(\mathbf{K}_0)(\sigma_{\max}^\mathbf{CQC^T+R}+||CA||^2\mathcal{A}(\mathbf{K}^*))}{\sigma_{\min}^\Sigma\sigma_{\min}^\mathbf{R}}$.
\end{proof}

\section{Discussions}

This paper is on learning the KF by optimization theory and {high dimensional statistics}. Compared with existing  asymptotic results of adaptive filtering \cite{1100694,1624478,ODELSON2006303,9044358}, we introduce a SGD method for a finite-horizon KF problem, and investigate the sample complexity of non-asymptotic error bound and the global convergence of the proposed SGD method. 

Recently, there have been many researches on handling KF problems by optimization theory and high dimensional probability. Specifically, {In \cite{10156641}}, a steady-state KF problem is investigated in the setting of both unknown model parameters and unknown noise covariance matrices. By solving receding-horizon KF problems with SGD method and zeroth-order optimization technique, the  steady-state optimal solution to infinite-horizon KF problem is approximated and the sample complexity of this framework is also studied \cite{10156641}. 
Compared with \cite{10156641}, we propose a SGD method to solve a  finite-horizon KF problem in the setting of unknown noise covariance matrices and yet known system parameters. 
%
%Inspired by \cite{10156641}, 
%
Specifically, we introduce a cost function on the observation and prove that the minimizer to this cost function is the Kalman gain; we further construct an unbiased estimation of gradient by using the model parameters and observations, and show the sample complexity and convergence rate of the resulting SGD method. 
The results of sample complexity indicates that our estimation method needs less samples than that of zeroth-order optimization technique \cite{10156641}.

The work \cite{Talebi2023DatadrivenOF} examines learning the steady-state KF by a SGD method in the setting of known system parameters and unknown noise covariance matrices. Using the duality between control and estimation, the KF problem is  reformulated as an optimal control problem of the adjoint system, and a SGD method is introduced on this landscape; by the approach of high dimensional statistics, the non-asymptotic error bound guarantee of the proposed SGD method is presented under the assumption that the noise is bounded. The results of this paper can be viewed as an extension of \cite{Talebi2023DatadrivenOF} to the setting of handling finite-horizon KF problem, and there are mainly two facts that make our results differ from those of \cite{Talebi2023DatadrivenOF}.
Firstly, the KF is time-dependent in our setting, and we reformulate the learning problem as an optimal control problem of the dual system, which differs from the deterministic one for the backward adjoint system of \cite{Talebi2023DatadrivenOF}. %The duality of this paper is used to analyze the equivalence between the proposed optimal control problem and the finite-horizon KF problem, and yet the duality in \cite{Talebi2023DatadrivenOF} is used to propose the SGD method.
Secondly, we consider the scenario that system noise and observation noise are both subgaussian; though explicit results are just given for the case of Gaussian noise, they can be easily extended to the setting of subgaussian noise by changing some coefficients only. Note that a bounded random variable is subgaussian and the work  \cite{Talebi2023DatadrivenOF} only considers the scenario of bonded noises. Interestingly, the non-asymptotic error bound of \cite{Talebi2023DatadrivenOF} depends explicitly on the bound of the noises.

 %which contains the scenario that noise is bounded, and give the . The sample complexity of our method is same with \cite{Talebi2023DatadrivenOF}.    
%W This problem is not be discussed in detail in the past, we formulate the learning problem into an optimization problem inspired by \cite{10156641} and \cite{Talebi2023DatadrivenOF}, and transform the optimization problem into the form of an optimal control problem. Then we analyze the global linear convergence guarantee of exact GD method, the analysis is not essentially different from analysis of \cite{doi:10.1137/20M1382386}. The main contribution of this paper is that we do not use zeroth-order optimization method to construct an biased estimation of gradient like other PG method\cite{doi:10.1137/20M1382386}, using high dimensional statistics and model parameters, we propose an unbiased estimation of gradient and analyze the convergence rate and sample complexity of SGD method  based on this unbiased estimation.

%\cite{10156641} 
\section{Numerical Example}

In this section, we give some numerical simulations to verify the effectiveness of our method. Let the system parameters and initial state be
\begin{equation}
	\nonumber
	\begin{aligned}
		&A=\begin{bmatrix}
			0.24& -0.18& -0.3118\\
			-0.0578& 0.4839& -0.0279\\
			-0.1283& -0.0138& 0.4761\\
		\end{bmatrix},C=\begin{bmatrix}
			0& 0.7071& 1.2247\\
			0.7071& -0.5125& 1.1124\\
		\end{bmatrix},\\&Q=\begin{bmatrix}
			0.61& -0.195& -0.3377\\
			-0.195& 0.775& -0.0953\\
			-0.3377&  -0.0953&   0.665\\
		\end{bmatrix},R=\begin{bmatrix}
			0.9& 0\\
			0& 0.6\\
		\end{bmatrix}, P_0=\begin{bmatrix}
		0& 0& 0\\
		0& 0& 0\\
		0& 0& 0\\
	\end{bmatrix}, x_0=\begin{bmatrix}
	1\\
	1\\
	0\\
\end{bmatrix}
	\end{aligned}
\end{equation}
with $Q_t\equiv Q, R_t\equiv R, t\in\mathbb{N}$. The normalized estimation error is defined as $\frac{f(\mathbf{K}_k)-f(\mathbf{K}^*)}{f(\mathbf{K}^*)}$.

We first verify the convergence of exact GD method. Letting $M=3$, $\mathbf{K}_0=0_{3\times 6}$, step size $\eta=0.0008$, and number of iterations $V=1000$, the numerical result of running exact GD method is shown in Figure \ref{fig:1}; it shows that $f(\mathbf{K}_k)$ converges to $f(\mathbf{K}^*)$ with fixed step size.
\begin{figure}[hp]
	\centering
	\includegraphics[scale=0.6]{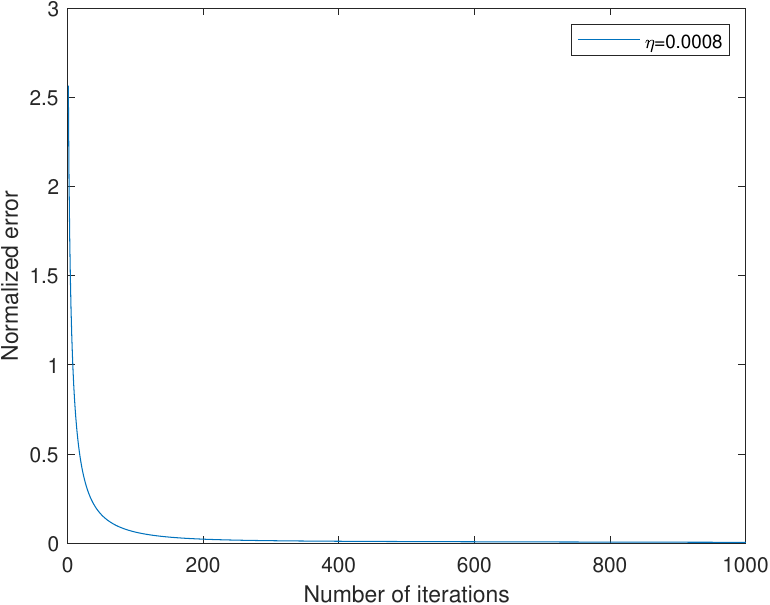}
	\caption{Convergence of GD}
	\label{fig:1}
\end{figure}

For the setting of unknown covariance matrices $Q$ and $R$, we use Algorithm 1 to compute $K_t^*, t\in \mathbb{T}$ with $M=3$. Let initial policy $\mathbf{K}_0=0_{3\times 6}$, step size $\eta=0.0008$, number of iterations $V=4000$, and number of samples $L=200$. The numerical result is shown in Figure \ref{fig:2} and the average process of normalized estimation error over 10 simulations converges to $0$ around. If we set the number of samples $L$ be $2000$, the ultimate normalized estimation error and fluctuation range will be smaller as shown in Figure \ref{fig:3}. This indicates that  $\mathbf{K}_k$ will converge to $\mathbf{K}^*$ with enough samples.

\begin{figure}[h]
	\centering
	\subfigure[$L=200$]{
		\includegraphics[scale=0.5]{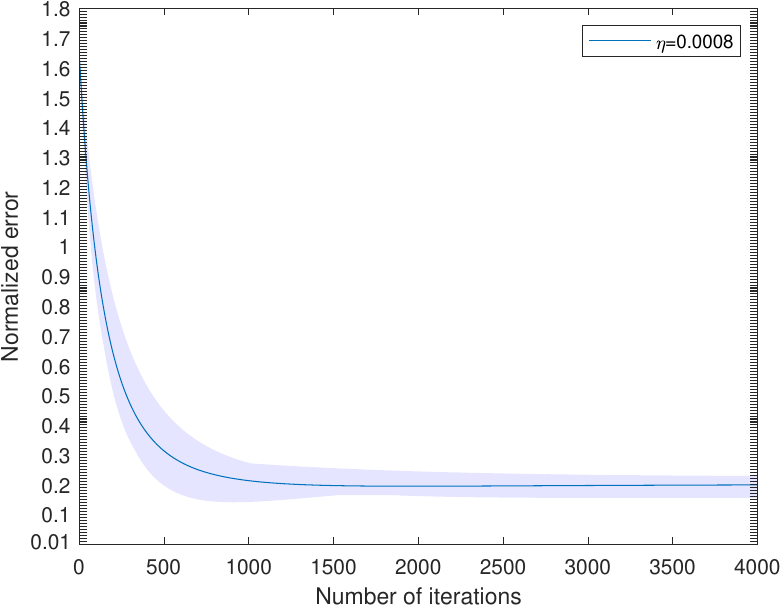}
		\label{fig:2}
		}
	\subfigure[$L=2000$]{
	\includegraphics[scale=0.5]{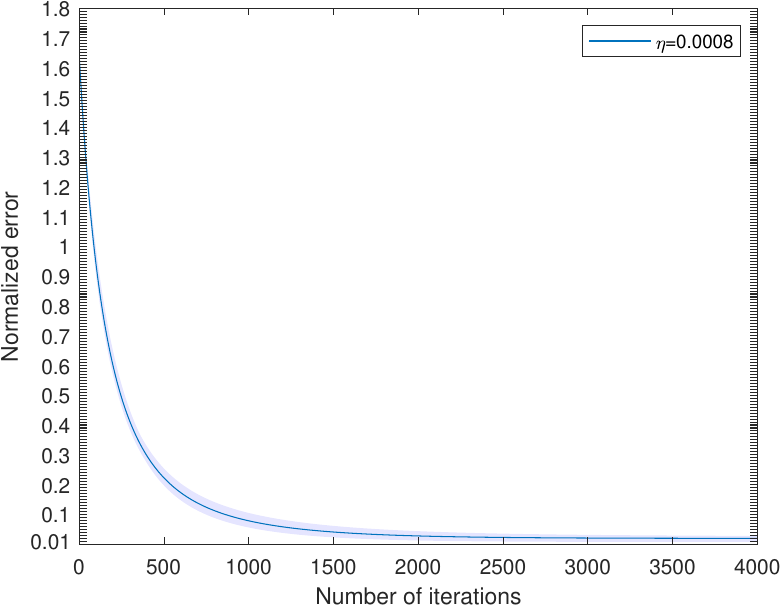}
\label{fig:3}}
\caption{The convergence of SGD over 10 simulations with different numbers of samples. The solid blue line is the average process of normalized estimation error over 10 simulations, and the shaded  light blue area is the fluctuation range of normalized estimation error over 10 simulations. %The error of average process and fluctuation range decrease when number of samples increase.
}
\label{Fig.main}
\end{figure}

Consider the scenario that $Q_t$ changes with time slowly
\begin{equation}
	\nonumber
	\begin{aligned}
		\delta Q=\begin{bmatrix}
			0.12& -0.08& 0\\
			-0.08& 0.12& 0\\
			0&   0&   0.05\\
		\end{bmatrix},
		Q_t=\begin{bmatrix}
			0.61& -0.195& -0.3377\\
           -0.195& 0.775& -0.0953\\
           -0.3377&  -0.953&   0.665\\
		\end{bmatrix}+t*\delta Q,~~~t=0,\dots,M+N.
	\end{aligned}
\end{equation}
Set initial policy $\mathbf{K}_0=\{K^0_0,K^0_1,K^0_2\}=0_{3\times 6}$, step size $\eta=0.0008$, and number of iterations $V=4000$.
Figure \ref{fig:4} and Figure \ref{fig:5} show that the average process of normalized estimation error over 10 simulations converges to $f(\mathbf{K}^*)$ around with number of samples $L=200$ and $L=2000$, respectively. %It indicates that $\mathbf{K}_k$ can also monotonically converge to $\mathbf{K^*}$ when $Q_t$ changes with time slowly. 
Comparing Figure \ref{Fig.main} with \ref{Fig.main2}, the estimation error in the setting of time-varying $Q_t$ is bigger than that in the setting of time-invariant $Q_t$. 

\begin{figure}[hp]
	\centering
	\subfigure[$L=200$]{
		\includegraphics[scale=0.5]{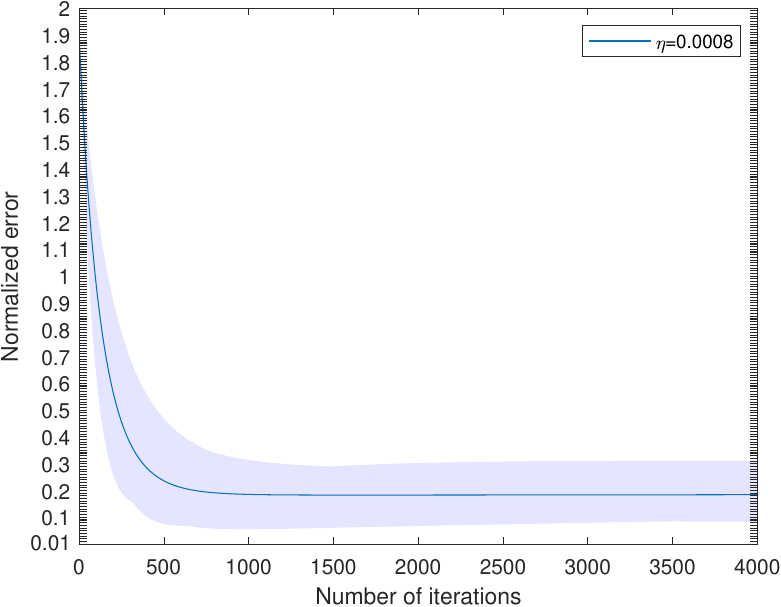}
		\label{fig:4}
	}
	\subfigure[$L=2000$]{
		\includegraphics[scale=0.5]{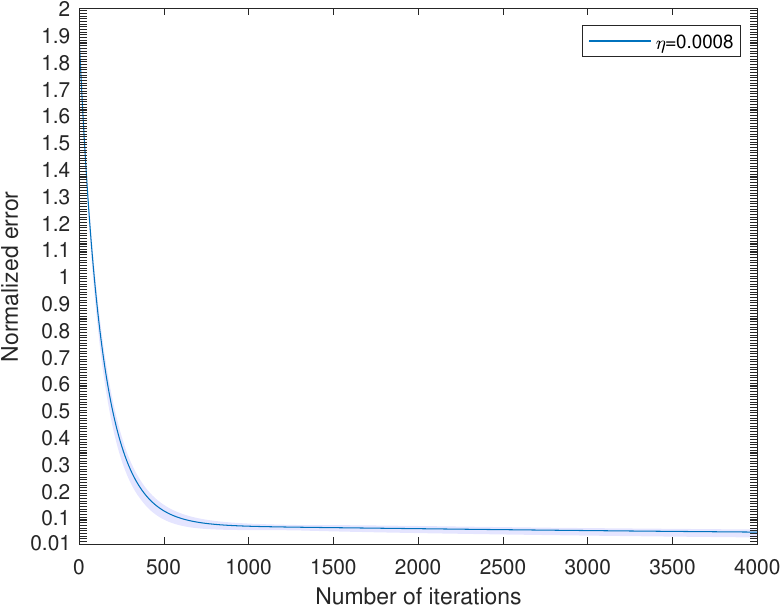}
		\label{fig:5}}
	\caption{The convergence of SGD over 10 simulations with different numbers of samples with time-varying $Q_t$. }
	\label{Fig.main2}
\end{figure}

\section{Conclusion}

In this work, we consider to solve the finite-horizon KF problem with unknown noise covariance, which is reformulated as a policy optimization problem of the dual system. We provide the global convergence guarantee and sample complexity of the proposed SGD method. Our consideration might be viewed as a supplement to the adaptive filtering methodology. 
In the future, we will explore the finite-horizon KF problem and steady-state KF problem for more general systems.

\bibliographystyle{unsrt}
\bibliography{reference}
\bibliographystyle{plainnat}
\vspace{12pt}
\end{document}